\documentclass[12pt]{amsart}
\usepackage{psfrag,epsfig,amssymb,url}
\numberwithin{figure}{section}
\numberwithin{table}{section}

\newtheorem{theorem}{Theorem}[section]

\newtheorem{lemma}[theorem]{Lemma}
\newtheorem{prop}[theorem]{Proposition}

\newtheorem{cor}[theorem]{Corollary}
\theoremstyle{definition}
\newtheorem{definition}[theorem]{Definition}
\newtheorem{example}[theorem]{Example}

\theoremstyle{remark}
\newtheorem{remark}[theorem]{Remark}
\numberwithin{equation}{section}

\def \R{{\mathbb R}}

\def \h{{\mathfrak h}}
\def \Z{{\mathbb Z}}
\def \S{{\mathbb S}}

\def \[{[ }
\def \]{] }

\def \L{{\mathcal L}}
\def \P{\r{P}}

\def \SS{\mathcal S}

\def \P{\mathfrak P}
\def \DP{{\mathfrak D\mathfrak P}}
\def \ADP{{\mathfrak A\mathfrak D\mathfrak P}}

\def \Top{{\mathfrak T\mathfrak o\mathfrak p}}
\def \TopDP{{\mathfrak T\mathfrak o\mathfrak p\mathfrak D\mathfrak P}}

\begin{document}

\author{Anna Felikson}
\address{Independent University of Moscow}
\curraddr{Jacobs University Bremen}
\email{a.felikson@jacobs-university.de}
\thanks{Research of the second author is supported by grants RFBR 10-01-00678 and NSh 709.2008.1}

\author{Sergey Natanzon}
\address{Higher School of Economics, Moscow  \phantom{wwwwwwwwwwwwwwwwwwwwwwwwwwww}
Belozersky Institute of Physico-Chemical Biology, Moscow State University,\phantom{wwwww}
Institute Theoretical and Experimental Physics}
\email{natanzons@mail.ru}



\title{Double pants decompositions of 2-surfaces}

\begin{abstract} 
We consider a union of two pants decompositions of the same orientable 2-dimensional surface of any genus $g$. 
Each pants decomposition corresponds to a 
handlebody bounded by this surface, so two pants decompositions correspond to a Heegaard splitting of a 3-manifold.
We introduce a groupoid $FT$ acting on double pants decompositions. 
This groupoid is generated by two simple transformations (called flips and handle twists),
 each transformation changing only one curve of the double pants decomposition.
We prove that $FT$ acts transitively on all double pants
decompositions corresponding to Heegaard splittings of a 3-dimensional sphere.
As a corollary, we prove that the mapping class group of the surface is contained in $FT$.

\end{abstract}


\maketitle

\tableofcontents

\section*{Introduction}

Consider a 2-dimensional orientable surface $S$ of genus $g$ with $n$ holes. 
A pants decomposition of $S$ is a decomposition into 3-holed spheres  
(called ``pairs of pants''). For each pants decomposition $P$ of $S$ one may construct a handlebody $S_+$ such that $S$ is 
the boundary of $S_+$  and all curves of $P$ are contractible inside $S_+$. A union of two pants 
decompositions of the same surface define two different handlebodies bounded by $S$, attaching these
 handlebodies along $S$  one obtains a Heegaard splitting of some 3-manifold.

Below we define double pants decompositions of surfaces as a union of two pants decompositions (with an additional assumption
that the homology classes of the curves contained in the double pants decomposition generate the entire homology lattice $H_1(S,\Z)$).
We introduce a simple groupoid acting on double pants decompositions and prove that this groupoid acts transitively on all double pants
 decompositions (of the same surface) resulting in Heegaard splitting of a 3-sphere.

More precisely,  
there are two simple invertible operations acting on the set of double pants decompositions.
 One of these operations is a flip (or Whitehead move)
performed in any of the decompositions. Another operation is called a ``handle twist'', it is performed only if a handle is cut out by 
a curve contained in both pants decompositions. In this case the handle contains exactly one curve from each of 
the two pants decompositions, and one applies a Dehn twist along one of them to the other.

\medskip
\noindent
{\bf Main Theorem.}
{\it Let $S$ be a topological 2-surface of genus $g$ with $n$ holes,
where  $2g+n> 2$.
Then the  groupoid generated by flips and handle twists acts transitively on
the set of all double pants decompositions of $S$ corresponding to Heegaard splittings of 3-sphere}. 

\medskip

For the case of surfaces with holes one defines handlebodies with  marked 2-disks on the boundary
and glues two handlebodies so that the marked disks are attached to marked disks.
Slightly refining the definition of a pair of pants one may also substitute all or some of the holes by marked points.

It is possible to show that the mapping class group acts effectively on double pants decompositions 
of some special combinatorial class. Together with the Main Theorem this implies the following result: 

\medskip
\noindent
{\bf Corollary.}
{\it The category of double pants decompositions of a topological surface 
with morphisms generated by flips and handle twists
contains a subcategory isomorphic to a category of topological surfaces with the mapping class group as the set of morphisms}.

\medskip

The idea to consider two pants decompositions as a Heegaard splitting of a 3-manifold already appeared in the literature
(\cite{CG},~\cite{He},~\cite{luo}).
The main tool of study there is the curve complex defined by Harvey~\cite{Har}, which is a simplicial complex whose 
 vertices are homotopy classes of closed curves on the surface and simplices are spanned by mutually disjoint curves. 
In particular, Hempel~\cite{He} defines a distance between two pants decompositions in terms of the distance in the curve complex
and uses this distance to describe topological properties of the resulting 3-manifolds.
Our construction is also naturally related to the pants complex described by Hatcher and Thurston (\cite{HT},~\cite{H1})
and to investigation of the mapping class group through various complexes (\cite{ABP},~\cite{BK},~\cite{I},~\cite{IK},~\cite{M}
and many others).
The notion of a groupoid generated by flips and twists on the Heegaard splittings seems to be new.

The paper is organized as follows. 
In Section~\ref{def} we discuss properties of ordinary pants decompositions. 
In Section~\ref{secDP} we introduce double pants decompositions, admissible double pants decompositions and the groupoid 
generated by flips and twists.
We prefer to work 
with surfaces containing neither holes nor marked points and postpone all details concerning the holes (marked points) till
Section~\ref{s open}.
In Section~\ref{g=2}, we prove transitivity theorem for the case of surfaces of genus $g=2$ without marked points (Theorem~\ref{g=2}). 
In Section~\ref{g>2}, we prove the  Main Theorem for the case of surfaces of any genus (Theorem~\ref{trans}). 
In Section~\ref{s open}, we extend basic definitions to the case of surfaces with holes or marked points and
complete  the proof of Main Theorem.
Finally, in Section~\ref{sec top},  we prove the Corollary (Theorem~\ref{top}).

\medskip


\medskip
\noindent
{\bf Acknowledgments.} We are grateful to Allen Hatcher, Saul Schleimer,
 Anton Zorich, Karen Vogtmann  and Ursula Hamenst\"adt for many helpful conversations. 
The work was initiated during the authors' stay at Max Planck Institute for Mathematics and completed during the stay of the first 
author at the same Institute. We are grateful to the Institute for hospitality, support and a nice working atmosphere.
Finally, we are grateful to the anonymous referee for a bunch of valuable suggestions.

\section{Pants decompositions}
\label{def}

\subsection{Zipped pants decompositions}

Let $S=S_{g,n}$ be an oriented surface of genus $g$ with $n$ or marked points, 
where $2g+n>2$.
Throughout Sections~\ref{def}--\ref{g>2} we assume in addition $n=0$ (this assumption will be removed in   
Section~\ref{s open}). 


A {\it curve} $c$ on $S$ is an embedded closed non-contractible curve considered up to a homotopy of $S$. 

Given a set of curves we always assume that there are no ``unnecessary intersections'', so that if two curves of this set
intersect each other in $k$ points then there are no homotopy equivalent pair of curves intersecting in less than $k$ points.

For a pair of curves $c_1$ and $c_2$ we denote by $|c_1\cap c_2|$ the geometric number of intersections 
(i.e. counting without multiplicities) of $c_1$ with $c_2$.

\begin{definition}[{{\it Pants decomposition}}]
A {\it pants decomposition} $P$ of $S$ is a system of nonoriented mutually disjoint curves 
$P=P_S=\{ c_1,\dots,c_k \}$  decomposing  $S$ into pairs of pants 
(i.e. into spheres with 3 holes).

\end{definition}

It is easy to see that any pants decomposition of a surface of genus $g$ consists of $3g-3$ curves.
Note, that we do allow self-folded pants, two of whose boundary components are identified in $S$ as shown in Fig.~\ref{pants}.

A surface which consists of one self-folded pair of pants will be called {\it handle}.

We say that a curve $c\in P$ is {\it non-regular} if $c$ is contained in a handle $\h$ such that the boundary $c_1$ of $\h$
 is also contained in $P$ 
(see Fig.~\ref{pants}). Otherwise, we say that $c$   is {\it regular}.

\begin{remark}
A set of curves forming a pants decomposition is maximal in sense that  any larger set of mutually disjoint non-oriented curves 
contains  homotopy equivalent curves.

\end{remark}

\begin{figure}[!h]
\begin{center}
\psfrag{c}{\scriptsize $c$} 
\psfrag{c1}{\scriptsize $c_1$} 
\psfrag{a}{\small (a)}
\psfrag{b}{\small (b)}
\epsfig{file=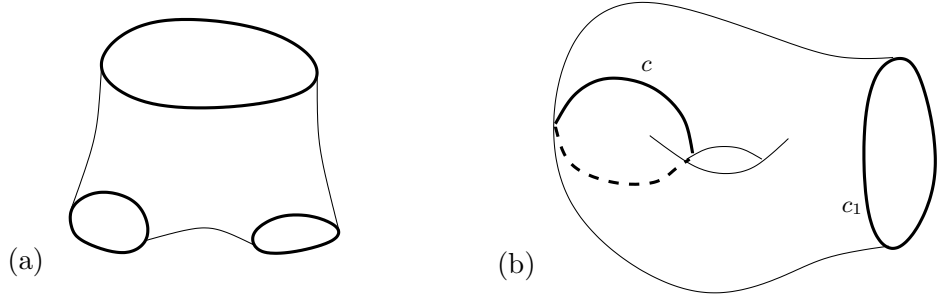,width=0.8\linewidth}
\caption{(a) A pair of pants; (b) a pair of self-folded pants composing a handle (the handle contains a  non-regular curve $c$).} 
\label{pants}
\end{center}
\end{figure}

\begin{definition}[{{\it Zipper system}}]
\label{zip}
A union $Z=\{ z_1,\dots z_{g+1} \}$  of mutually disjoint curves is a {\it zipper system}
if $Z$ decomposes $S$ into two spheres with $g+1$ holes.

\end{definition}

\begin{definition}[{{\it $Z$ compatible with $P$}}]
A zipper system $Z$ is {\it compatible with a pants decomposition} $P=\{ c_1,\dots,c_{3g-3}\}$ 
if $|c_i\bigcap (\cup_{j=1}^{g+1}z_j)|=2$ for each $i=1,\dots,3g-3$.

\end{definition}

Fig.~\ref{zip-ex} contains an example of a zipper system $Z$ compatible with a pants decomposition $P$.

\begin{figure}[!h]
\begin{center}
\psfrag{P}{\small $P$}
\psfrag{Z}{\small $Z$}
\epsfig{file=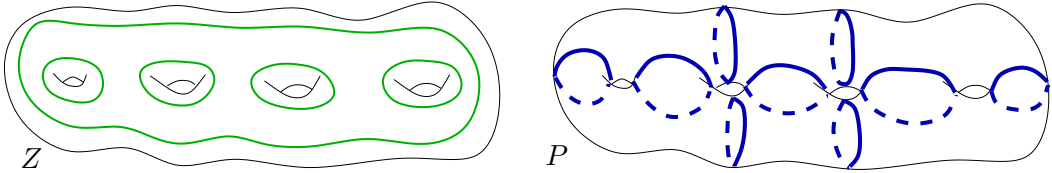,width=0.9\linewidth}
\caption{ A zipper system $Z$ compatible with a pants decomposition $P$.} 
\label{zip-ex}
\end{center}
\end{figure}

\begin{lemma}
\label{compatible zipper}
If $Z$ is a zipper system compatible with  a pants decomposition $P$ then 
$\cup_{j=1}^{g+1}z_j$ decomposes each pair of pants in $P$ into two hexagons.

\end{lemma}

\begin{proof}
Suppose that a curve $z_j$ intersects a curve $c_i$ contained in the boundary of a pair of pants $p_1$.
The curves of the pants decomposition cut $z_j$ into segments.
Let $l$ be a segment of $z_j$ (or one of such segments) contained in $p_1$.
Since curves do not have unnecessary intersections, $l$ looks as shown in Fig.~\ref{segment}(a) or (b).
If $l$ looks as  in Fig.~\ref{segment}(a) then for some of the three boundary curves of $p_1$ the condition
 $|c_i\bigcap(\cup_{j=1}^{g+1}z_j)|=2$ is broken.
This implies that it is as one shown in Fig.~\ref{segment}(b).  
Therefore, $p_1$ looks like in  Fig.~\ref{segment}(c),
i.e. $p_1$ is decomposed into two hexagons. 

\end{proof}

\begin{remark}
\label{involution}
If $Z$ is a zipper system compatible with  a pants decomposition $P$ then there exists an orientation reversing involution 
$\sigma$ of $S$ such that 
$\sigma$ preserves $Z$ pointwise and such that $\sigma(c_i)=c_i$ for each $c_i\in P$. To build this involution one needs only to switch 
the pairs of hexagons described in Lemma~\ref{compatible zipper}. 

\end{remark}

\begin{figure}[!h]
\begin{center}
\psfrag{a}{\small (a)}
\psfrag{b}{\small (b)}
\psfrag{c}{\small (c)}
\epsfig{file=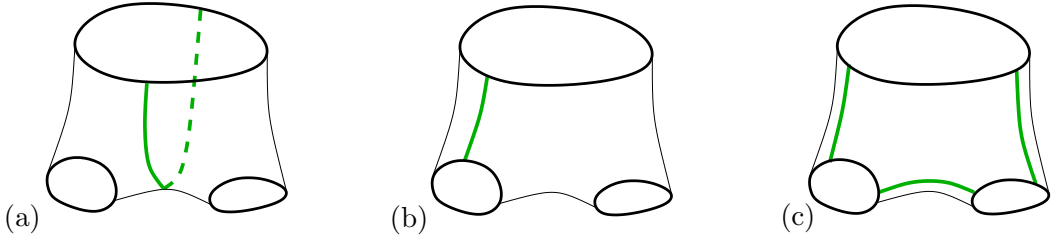,width=0.9\linewidth}
\caption{ Zipper system $Z$ decomposes each pair of pants into two hexagons.}
\label{segment}
\end{center}
\end{figure}

\begin{definition}[{{\it Zipped flip}}]
\label{flip}
Given a pants decomposition $P=\{ c_1,\dots,c_{3g-3} \}$ and a zipper system $Z$ compatible with $P$ we define a 
 {\it zipped flip} of pants decomposition as it is shown in Fig.~\ref{oriented flip}.
Formally speaking,  
a {\it zipped flip} $f_i$ of a pants decomposition $P=\{ c_1,\dots,c_{3g-3} \}$  {\it in the curve $c_i$ }is a
replacing  a regular curve $c_i\subset P$ by a unique curve $c_i'$ satisfying the following properties:
\begin{itemize}
\item $c_i'$ does not coincide with any of $c_1,\dots,c_{3g-3}$;
\item $c_i'$ intersects $Z$ exactly in two points;
\item $c_i'\cap c_j=\emptyset$ for all $j\ne i$.
\end{itemize}

\end{definition}

Clearly, $f_i(P)$ is a new pants decomposition of $S$, and $Z$ is a zipper system compatible with  $f_i(P)$.
The uniqueness of the curve $c_i'$ satisfying the properties in Definition~\ref{flip} verifies trivially.
In particular, it is easy to see that $f_i\circ f_i(c_i)=c_i$.

\begin{figure}[!h]
\begin{center}
\psfrag{f}{\scriptsize $f_i$}
\psfrag{c}{\scriptsize $c_i$}
\psfrag{c'}{\scriptsize $c_i'$}
\epsfig{file=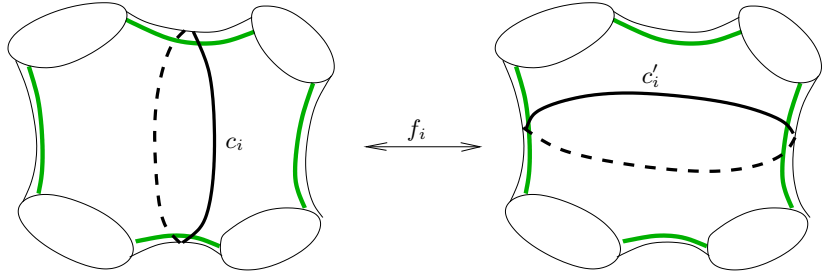,width=0.7\linewidth}
\caption{A zipped flip of a pants decomposition.} 
\label{oriented flip}
\end{center}
\end{figure}

\subsection{Unzipped pants decompositions}

\begin{definition}[{{\it Lagrangian plane of pants decomposition}}]
Let $P=\{ c_1,\dots,c_{3g-3}\}$ be a pants decomposition. A Lagrangian plane  $\L(P)\subset H_1(S,\Z)$ is a sublattice 
spanned by the homology classes $h(c_i)$, $i=1,\dots,3g-3$ (here $c_i$ is taken with any orientation).

\end{definition}

\medskip
\noindent
{\bf Notation.} By $\langle h(c_1,),\dots,h(c_k) \rangle$ we denote the sublattice of $H_1(S,\Z)$ spanned by
$h(c_1),\dots,h(c_k)$, so that $\L(P)=\langle h(c_1,),\dots,h(c_{3g-3}) \rangle$.
 
\medskip

\begin{definition}[{{\it Unzipped flip}}]
Let $P=\{ c_1,\dots,c_n  \}$ be a pants decomposition.
Define an {\it unzipped flip  of $P$ in the curve} $c_i$ (or just a {\it flip}) as  
a replacing  a regular curve $c_i\subset P$ by any curve $c_i'$ satisfying the following properties:
\begin{itemize}
\item $c_i'$ does not coincide with any of $c_1,\dots,c_n$;
\item $|c_i'\cap c_i|=2$;
\item $c_i'\cap c_j=\emptyset$ for all $j\ne i$.
\end{itemize}

\end{definition}

\begin{prop}
 $\L(P)=\L(f(P))$ for any unzipped flip $f$ of $P$.

\end{prop}

\begin{proof}
Let $p_1\cup p_2$ be two pairs of pants glued along a curve $e$ changed under the flip $f$.
Let $a,b,c,d$ be   four curves  cutting $p_1\cup p_2$ out of $S$.
Suppose that $f(e)$ separates $a$ and $b$ from $c$ and $d$ in  $p_1\cup p_2$. 
Denote by $h(c)$ the homology class of $c$.
Then 
$$ h(f(e))= h(a)+h(b)=h(c)+h(d)\in \L(P).$$
The similar relation holds when $f(e)$ separates $a$ and $c$ from $b$ and $d$ or $a$ and $d$ from $b$ and $c$. 

\end{proof}

\begin{lemma}
\label{flip modulo twist}
Let $P$ be a pants decomposition and $c\in P$ be a curve. A flip $f$ of $P$ in the curve $c$ is defined uniquely by a homology class of $f(c)$ up to a Dehn twist along $c$.

\end{lemma}

\begin{proof}
Consider two pairs of pants $p_1$ and $p_2$ adjacent to $c$, let $M=p_1\cup p_2 $. Since $|(f(c)\cap c)|=2$ and 
$f(c)\cap \partial M=\emptyset$, the segment $f(c)\cap p_i$, $i=1,2$ looks as shown in Fig.~\ref{segment}(a). 
The homology class of $f(c)$ defines which of the ends of $f(c)\cap p_1$ are glued to which of the ends of $f(c)\cap p_2$.   
So, the only freedom in gluing of $p_1$ to $p_2$ is generated by a Dehn twist along $c$. 

\end{proof}

The result of Lemma~\ref{flip modulo twist} may be restated as follows.

\begin{lemma}
\label{zip for flip}
Let $P$ be a pants decomposition compatible with a zipper system $Z$. Let $f_c$ be a flip of $P$ in a curve $c\in P$.
Then $f_c$ is a zipped flip for some $Z'=T_c^m(Z)$, where  $m\in \Z$ and $T_c$ is a Dehn twist along $c$.  

\end{lemma}

\begin{remark}
\label{adjacent zips}
In particular, Lemma~\ref{zip for flip} implies that if $P_0$ is a pants decomposition compatible with a zipper system $Z$
then after any sequence of flips $f_1,\dots, f_k$ we obtain a decomposition $P_k$ compatible with some zipper system $Z_k$.
Moreover, one can choose zipper systems $Z_1,\dots,Z_k$, $Z_0=Z$ so that $f_i$ is a  zipped flip with respect to $Z_i$
and $Z_{i+1}=T_{c_i}^{m_i}(Z_i)$  for the curve $c_i\in P_i$ changed by $f_i$.  

\end{remark}

A Dehn twist along a curve $c\in P$ is a composition of two flips (see  Fig.~\ref{dehn}).

\begin{figure}[!h]
\begin{center}
\psfrag{c}{\scriptsize $c$}
\epsfig{file=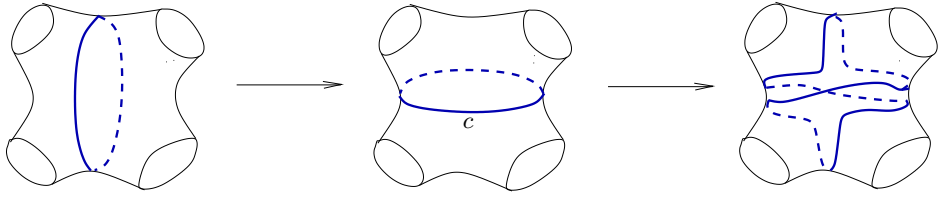,width=0.8\linewidth}
\caption{Dehn twist as a composition of two flips.} 
\label{dehn}
\end{center}
\end{figure}

\begin{definition}[{{\it Category of unzipped pants decompositions}}]
A category $\P_g(\L)$ for the given Lagrangian plane $\L$ is the following:

{\bf Objects:} pants decompositions  $P$ of a genus $g$ surface $S$ satisfying $\L(P)\in \L$;

{\bf Elementary morphisms:} 
unzipped flips.

Other {\bf morphisms} are compositions of elementary ones.

\end{definition}

\begin{example}
\label{S_2-orbit}
(An orbit of a pants decomposition of $S_2$.)
In this example we describe an orbit of an arbitrary  pants decomposition of a surface of genus 2.
The description is in terms of a graph $\Gamma$ where a vertex $v_P\in \Gamma$  correspond to a pants decompositions $P$ 
and an edge $e\in \Gamma$ connecting $v_P$ to $v_{P'}$  correspond to a flip $f$ such that $P'=f(P)$.
 
A pants decomposition $P$ of $S_2$ may be of one of two types:
\begin{itemize}
\item{(a)} ``non-self-folded'': $P$ consists of two non-self-folded pairs of pants (the pants decompositions of this type are
denoted by squares in Fig.~\ref{S2-orbit});

\item{(b)} ``self-folded'': $P$ consists of two handles glued along the holes (the pants decompositions of this type are
denoted by circles in Fig.~\ref{S2-orbit}).
 
\end{itemize}

For each of the two types of vertices of $\Gamma$  we need to understand how many edges are incident to the vertex.
In fact, the number of such edges is always infinite: if $c\in P$ is a regular curve, $T_c$ is a Dehn twist along $c$  
and $f(c)$ is a flip of the curve $c$ then
$T_c(f(c))$ is also a flip of $c$ (see Fig.~\ref{dehn}). On the other hand, Lemma~\ref{flip modulo twist}  states
that modulo Dehn twist $T_c$ there are exactly two possibilities for the flip $f(c)$.
Therefore, instead of $\Gamma$ we will draw the simplified graph $\overline \Gamma$ obtained from $\Gamma$ after 
factorizing by Dehn twists $T_c$ for each flipped curve $c$.
Then for each regular curve $c\in P$ there are exactly 2 edges emanating from vertex $v_P\in \overline \Gamma$. 
If $P$ is of non-self-folded type, there are 3 regular curves in $P$, so there are 6 edges incident to  $v_P\in \overline \Gamma$.
If $P$ is of self-folded type, there is a unique regular curve in $P$, so there are only two edges incident to  $v_P\in \overline \Gamma$.
The graph $\overline \Gamma$ is shown in Fig.~\ref{S2-orbit}.

We will say that a path $\gamma\in \overline \Gamma$ is {\it alternating} if any edge of $\gamma$ 
connects two vertices of different types.
It is easy to see that for each  path $\gamma\in \overline \Gamma$ there exists an alternating path $\gamma'\in \overline \Gamma$
with the same endpoints.
Indeed, each edge of $\overline \Gamma$ connecting two vertices of the same type (those are always ``square'' vertices)
may be substituted by an alternating path of two edges.
Since a Dehn twist along a curve of pants decompositions is a composition of two flips 
(each changing the type of a pants decomposition in case of $S_2$),
we obtain the same property for an arbitrary path in $\Gamma$:

\smallskip
{\it
 For each  path $\gamma\in \Gamma$ there exists an alternating path $\gamma'\in \Gamma$
with the same endpoints.
}

\end{example}

\begin{figure}[!h]
\begin{center}
\psfrag{1}{\scriptsize $1$}
\psfrag{2}{\scriptsize $2$}
\psfrag{3}{\scriptsize $3$}
\epsfig{file=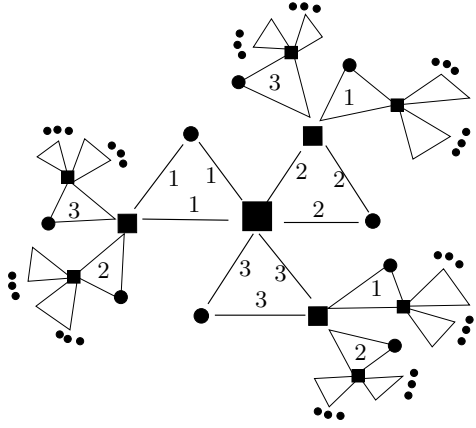,width=0.4\linewidth}
\caption{An orbit of a pants decomposition of $S_2$ (modulo Dehn twists). Vertices marked by squares and circle  correspond to 
non-self-folded and self-folded pants decompositions respectively. Edges correspond to flips.
The labels $1,2,3$ on the edges show which of the three curves of the decomposition is flipped along this edge.} 
\label{S2-orbit}
\end{center}
\end{figure}

Example~\ref{S_2-orbit} shows that the pants decompositions containing the curves separating handles play special role among other 
pants decompositions of a closed genus 2 surface. 

\begin{definition}[{{Standard pants decomposition}}]
A pants decomposition $P$ is {\it standard} if $P$ contains $g$ curves $c_1,\dots,c_g$ such that $c_i$ cuts out of $S$ a handle $\h_i$.

\end{definition}

\subsection{Pants decompositions and handlebodies}

A 3-dimensional {\it handlebody} is a 3-dimensional disk with several handles attached (where a  handle is a solid torus, and it is 
attached to a 3-disk along a 2-disk).

Given a pants decomposition $P$ one may construct a handlebody in the following way: for each pair of pants $P_i$ one considers a 
3-disk $D_i$ with three marked 2-disks on its boundary, then for each pair of adjacent (in $P$) pairs of pants $P_i$, $P_j$ one 
attaches $D_i$ to $D_j$ along the marked disks. One obtains a handlebody $H(P)$ whose boundary coincides with $S$ and carries the 
structure of the initial pants decomposition $P$.

The following proposition is evident:

\begin{prop}
Let $f$ be a flip of $P$. Then $H(f(P))=H(P)$.  

\end{prop}

\begin{prop}[A.~Hatcher,~\cite{Ha-priv}]
\label{trans on one plane}
{\it Morphisms of  $\P_g(\L)$ do not act transitively an the objects of the same category}.

\end{prop}

\begin{proof}
Suppose that the surface $S$ is embedded in $\R^3$ and a pants decomposition $P$ is such that each curve of $P$ is
contractible inside the inner handlebody defined by $S\subset \R^3$. Then each flip preserves this property of $P$.
On the other hand there exists a pants decomposition $P'\in \L(P)$ with non-contractible (inside the given handlebody) curves
(see Fig.~\ref{non-contractible} for a non-contractible curve $c$ such that $h(c)\in \L(P)$). 

\end{proof}

\begin{figure}[!h]
\begin{center}
\psfrag{1}{\scriptsize $c_1$}
\psfrag{2}{\scriptsize $c_2$}
\psfrag{3}{\scriptsize $c_3$}
\psfrag{c}{\scriptsize $c$}
\epsfig{file=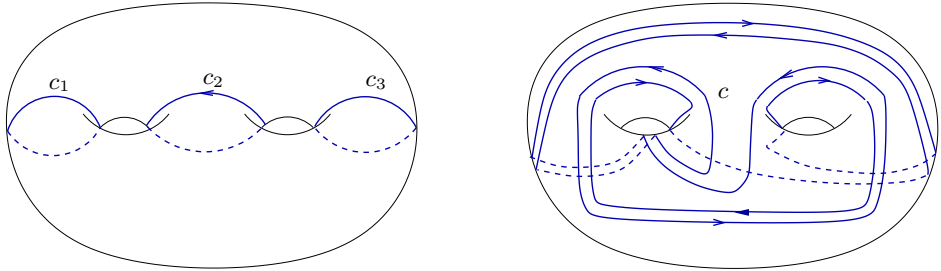,width=0.8\linewidth}
\caption{Pants decomposition $P=\{ c_1,c_2,c_3 \}$ and a curve $c$ such that $h(c)=h(c_2)$ and $c$ 
is not contractible inside the inner handlebody (the curve $c$ is linked non-trivially with $c_3$).} 
\label{non-contractible}
\end{center}
\end{figure}

\begin{prop}[F.~Luo,~\cite{luo}]
\label{handlebody}
Flips do act transitively on the pants decompositions whose curves are contractible in a given handlebody of a surface $S\subset \R^3$.

\end{prop}

\begin{cor}
\label{handle}
Flip-equivalence classes of pants decompositions of a given surface are in one-to-one correspondence with handlebodies 
defined by pants decompositions.

\end{cor}

\section{Double pants decompositions}
\label{secDP}

\subsection{Double pants decompositions and flip-twist groupoid}

\begin{definition}[{{\it Lagrangian planes in general position}}]
Two {\it  Lagrangian planes $\L_1$ and $\L_2$  are in general position} if 
$H_1(S,\Z)=\langle \L_1,\L_2\rangle$.

\end{definition}

See Fig.~\ref{double_pant} for an example of two pants decompositions spanning a pair of Lagrangian planes in general position.

\begin{figure}[!h]
\begin{center}
\psfrag{P}{$P_a$}
\psfrag{Z}{$P_b$}
\epsfig{file=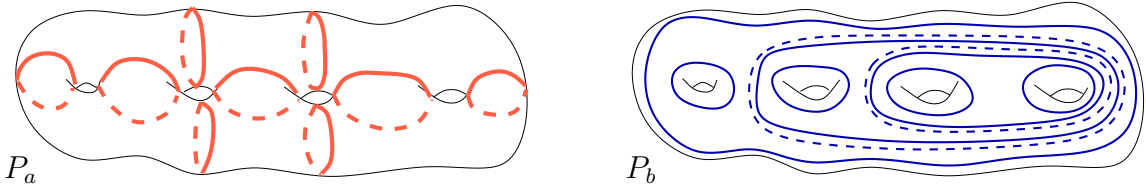,width=0.98\linewidth}
\caption{Pair of pants decompositions $(P_a,P_b)$. } 
\label{double_pant}
\end{center}
\end{figure}

\begin{definition}[{{\it Double pants decomposition}}]
A {\it double pants decomposition} $DP=(P_a,P_b)$ is a pair of pants decompositions $P_a$ and $P_b$ of the same surface
such that the Lagrangian planes $\L_a=\L(P_a)$ and  $\L_b=\L(P_b)$ spanned by these pants decompositions
are in general position.

\end{definition}

\begin{definition}[{{\it Handle twists}}]
\label{Dehn}
A {\it Handle twist} is a transformation of a double pants decomposition $DP=(P_a,P_b)$ 
which may be performed if $P_a$ and $P_b$ contain the same  curve $a_i=b_i$  
separating the same handle $\h$, 
see Fig.~\ref{d-self-pant}. Let $a\in \h$ and $b\in \h$ be the only curves of $P_a$ and $P_b$ contained in $\h$.
Then a {\it handle twist in $\h$}   is a Dehn twist along $a$ or along $b$ in any of two directions
(see Fig.~\ref{d-self-pant}(b)).  

\end{definition}

\begin{figure}[!h]
\begin{center}
\psfrag{a}{\scriptsize $a$}
\psfrag{b}{\scriptsize $b$}
\psfrag{a1}{\scriptsize $a'$}
\psfrag{b1}{\scriptsize $b'$}
\psfrag{a2}{\scriptsize $a''$}
\psfrag{b2}{\scriptsize $b''$}
\psfrag{aa}{\small (a)}
\psfrag{bb}{\small (b)}
\psfrag{cc}{\small (c)}
\psfrag{i}{\scriptsize $a_i=b_i$}
\epsfig{file=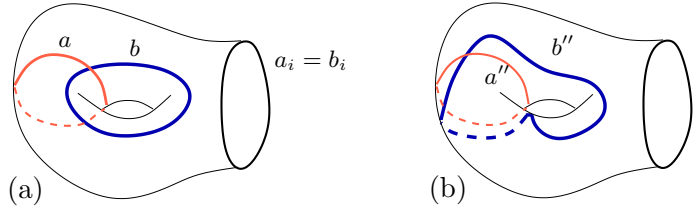,width=0.59\linewidth}
\caption{Handle twists:
(a) handle (double self-folded pair of pants); (b) the same handle after one of the four possible handle twists.} 
\label{d-self-pant}
\end{center}
\end{figure}

\begin{definition}[{{\it Category of double pants decompositions}}]
\label{def-double}
A category $\DP_{g,0}$ of {\it double pants decompositions} of a genus $g$ surface $S=S_{g,0}$ is the following:

{\bf Objects:} double pants decompositions $DP=(P_a,P_b)$ of $S$.

{\bf Elementary morphisms:} 
\begin{itemize}
\item unzipped flips of $P_i$ ($i\in \{ a,b\}$); 



\item handle twists. 

\end{itemize}

Other {\bf morphisms} are compositions of elementary ones.

\end{definition}

\begin{remark}
The index ``$g,0$'' in the notation $\DP_{g,0}$ is to underline that the objects of this category 
are double pants decompositions of surfaces of genus $g$ 
{\it without} marked points. 

\end{remark}

%

\begin{definition}[{{\it $\DP$-equivalence}}]
Two double pants decompositions are called $\DP$-equivalent if there exists a morphism of $\DP_{g,0}$ taking one of them to the other.

\end{definition}

\begin{definition}[{{\it Flip-twist groupoid}}]
All morphisms of $\DP_{g,0}$ are reversible, so the morphisms form a groupoid acting on the objects of $\DP_{g,0}$. 
We will call it a {\it flip-twist groupoid} and denote $FT$.

\end{definition}

\subsection{Admissible double pants decompositions}

\begin{definition}[{{\it Standard double pants decomposition, principle curves}}]
A double pants decomposition $(P_a,P_b)$ is {\it standard} if 
there exist $g$ curves $c_1,\dots,c_g$ such that
the following two conditions hold:
\begin{itemize}
\item $c_i\in P_a\cap P_b$;
\item $c_i$ cuts out of $S$ a handle $\h_i$. 
\end{itemize}

The set of curves  $\{c_1,\dots,c_g\}$ in this case is the {\it set of principle curves} of $(P_a,P_b)$. 
\end{definition}

See Fig.~\ref{standard double_pant} for an example of a standard double pants decomposition.
Notice that a standard double pants decomposition is not unique: many elements of the mapping class group of $S$ act non-trivially 
on the decomposition shown in  Fig.~\ref{standard double_pant}, moreover applying a flip to the curve in a middle of $P_a$ one
may obtain a decomposition which is combinatorially different from one shown in the figure (in the former there are $5$ curve contained
in $P_a\cap P_b$, while in the latter there are $4$ ones).

\begin{figure}[!h]
\begin{center}
\psfrag{a}{$P_a$}
\psfrag{b}{$P_b$}
\epsfig{file=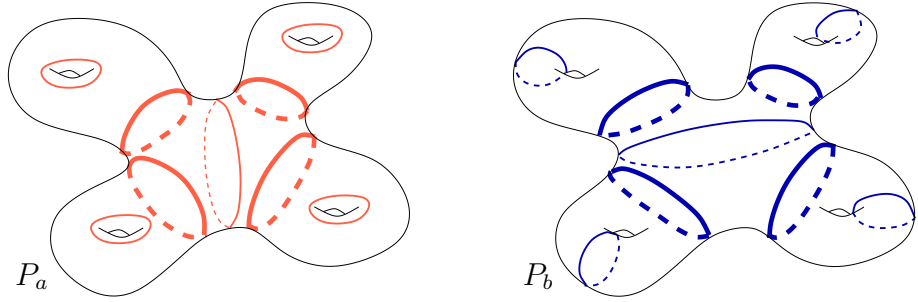,width=0.78\linewidth}
\caption{Standard double pants decomposition $(P_a,P_b)$. The principle curves are shown by bold lines. } 
\label{standard double_pant}
\end{center}
\end{figure}

%

\begin{definition}[{{\it Admissible decomposition}}]
\label{def admiss}
A double pants decomposition $(P_a,P_b)$ is {\it admissible}  if 
there exists a standard double pants decomposition
$(P'_a,P'_b)$ such that $P'_a$ is flip-equivalent to $P_a$ and $P_b'$ is flip-equivalent to $P_b$.

\end{definition} 

\begin{example}
It is easy to check that a double pants decomposition $(P_a,P_b)$ shown in Fig.~\ref{double_pant} is admissible.

\end{example}

\begin{lemma}
\label{ht-adm}
Let $DP$ be an admissible double pants decomposition and $t$ be a handle twist of $DP$. Then $t(DP)$ is also admissible.

\end{lemma}

\begin{proof}
Denote by $\bar t$ the Dehn twist of $S$, that acts on $DP$ as the handle twist $t$. Notice that applying $\bar t$ to a double 
pants decomposition we obtain the decomposition with the same combinatorial properties, in particular,  $\bar t$ takes flippable curves to 
flippable. Moreover, for each double pants decomposition $DP_0$ if $f$ is a flip of a curve $c\in DP_0$ then 
$f'=\bar t f \bar t^{-1}$ is a flip of the curve $\bar t(c)\in \bar t(DP_0)$
and $\bar tf\bar t^{-1}(\bar t(DP_0))= \bar t(f(DP_0))$.

Suppose now that $f_1,\dots,f_k$ is a sequence of flips taking $DP$ to a standard decomposition $DP'$.
Then the sequence $f_1',\dots,f_k'$, $f_i'=\bar tf_i\bar t^{-1}$ of flips takes $t(DP)$ to a standard decomposition $\bar t(DP')$.

\end{proof}

The set of admissible double pants decompositions is closed under the action of flips and handle twists, 
so we may define a subcategory of $\DP_{g,0}$:

\begin{definition}[{{\it Category of admissible double pants decompositions}}]
\label{def-adm-double}
A category $\ADP_{g,0}$ of {\it admissible double pants decompositions} of a genus $g$ surface  is the following:

{\bf Objects:} admissible double pants decompositions $DP=(P_a,P_b)$ of a genus $g$ surface.

{\bf Elementary morphisms:} 
\begin{itemize}
\item 
unzipped flips of $P_i$ ($i\in \{a,b\}$); 
\item 
handle twists. 

\end{itemize}

Other {\bf morphisms} are compositions of elementary ones.

\end{definition}

\subsection{Admissible decompositions as Heegaard splittings of $\S^3$}

Given a pants decomposition $P_a$ on $S$, one can define a handlebody $S_+$ such that all curves of $P_a$ are contractible in $S_+$.
A choice of two pants decompositions $P_a$ and $P_b$ gives  two handlebodies, which could be attached along $S$, 
so that we obtain a Heegaard splitting of a 3-manifold denoted by $M(P_a,P_b)$. 

\begin{lemma} 
\label{st is S}
If $(P_a,P_b)$ is a standard double pants decomposition then   $M(P_a,P_b)=\S^3$.

\end{lemma}

\begin{proof}
We consider the surface $S=S_{g,0}$ with a standard pants decomposition as a union of $g$ handles and a $g$-holed sphere.
Notice that on each of these surfaces the homology classes of the curves  $c\in P_a\cup P_b$ generate entire homology lattice.

First, we consider a handle (a one holed torus) $\h_i$  with double pants decomposition $DP=(P_a,P_b)\cap \h_i$. 
In this case $DP$ consists only of two 
curves $a_i$ and $b_i$. Since $\langle h(a_i),h(b_i)\rangle=H_1(\h,\Z)$,  $a_i$ intersects $b_i$ at a unique point. 
It is easy to see that
$\h$ may be embedded in $\S^3$ so that $a_i$ is contractible inside the inner solid torus (bounded by $ \h$ and a disk attached
to $\partial \h$ ) and $b_i$ is contractible outside one: 
to see that first we embed $a_i$ correctly and then apply several twists along $a_i$ to ``unwrap'' $b_i$.

Now, we embed each handle $\h_1,\dots,\h_g$ in the same $\S^3$ so that all images are disjoined. 
Then we attach the sphere $S_{0,g}$.
Since all curves on $S_{0,g}$ are homologically trivial, they are contractible both inside and outside.
So, we constructed an embedding of $S$ to $\S^3$ 
  such that all curves of $P_a$ and $P_b$ are contractible in inner and outer handlebodies respectively,
which is equivalent to $M(P_a,P_b)=\S^3$.

\end{proof}

\begin{lemma}
\label{S is adm}
If $M(P_a,P_b)=\S^3$ then $(P_a,P_b)$ is admissible. 

\end{lemma}

\begin{proof}
If  $M(P_a,P_b)=\S^3$ then there exists an embedding $S\hookrightarrow \S^3$ such that the curves of $P_a$ are contractible in inner 
handlebody for this embedding, while the curves of $P_b$ are contractible in the outer one. It is easy to find a standard double 
pants decomposition $(P_a',P_b')$ such that the curves of $P_a'$ and the curves of $P_b'$ are contractible in the inner and outer handlebodies respectively. In view of  Corollary~\ref{handle} flips act transitively on pants decompositions contractible inside a given handlebody, so the decomposition  $(P_a,P_b)$  is flip-equivalent to $(P_a',P_b')$, which is standard. So,  $(P_a,P_b)$ is admissible.

\end{proof}

Combining Lemma~\ref{st is S} with Lemma~\ref{S is adm} we arrive at the following characterization of admissible decompositions
(it also may be taken as an alternative definition):

\begin{theorem}
A double pants decomposition $(P_a,P_b)$ is admissible if and only if   $M(P_a,P_b)=\S^3$.

\end{theorem}

Our next aim is to prove that morphisms of $\ADP_{g,0}$ act transitively on the objects of $\ADP_{g,0}$.
In other words, we will prove that flips and twists act transitively on double pants decompositions corresponding to Heegaard splittings of $\S^3$.
This is done in Section~\ref{g=2} for the case of $g=2$ and in Section~\ref{g>2} for a general case.

\begin{remark}
An admissible double pants decomposition can be also defined as any decomposition obtained from a standard decomposition via a sequence of
 flips and handle twists. In principle, this class of decompositions may be wider than one in Definition~\ref{def admiss} 
(where only flips are allowed), so the transitivity theorem for this class
looks stronger than one we are going to prove. In fact, Lemma~\ref{ht-adm} shows that these two classes coincide.

\end{remark}

\section{Transitivity of morphisms in case $g=2$}
\label{g=2}

In this section we prove the Main Theorem  for the case of surface of genus $g=2$ containing no marked points. 
The proof is based on the following  result~\ref{teacher} of Hatcher and Thurston~\cite{HT}.

\begin{definition}[{{\it $\SS$-moves}}]
\label{S_}
Let $P$ be a pants decomposition of $S$ and $a,c\in P$ be two curves such that $c$ cuts out of $S$ a handle $\h$ and let $a\subset \h$.
An {\it $\SS$-move} of a pants decomposition $P$ in a curve $a$ is a substitution of $a$ by a curve $a'$, where  
$a'\in \h$ is an arbitrary curve such that $|a\cap a'|=1$.

\end{definition}

\begin{theorem}[A.~Hatcher, W.~Thurston~\cite{HT},~\cite{H1}]
\label{teacher}
Let $S_{g,n}$ be a surface of genus $g$ with $n$ holes.
Any pants decomposition of $S_{g,n}$ can be transformed to any other pants decomposition of $S_{g,n}$ via flips and
  $\SS$-moves.

\end{theorem}

\begin{remark}
In the initial paper of  Hatcher and Thurston~\cite{HT} the surface $S_{g,n}$ is supposed to be closed surface containing 
no marked points. This assumption is removed in~\cite{H1}. 

\end{remark}

\begin{remark}
Theorem~\ref{teacher} does not imply immediately transitivity of morphisms in $\ADP_{g,0}$
(Theorem~\ref{trans}) since the set of handle twists  in $\ADP_{g,0}$
is much smaller than the set of $\SS$-moves in Theorem~\ref{teacher} (the former depends on the relative position of two decompositions).

\end{remark}



Now we will prove several lemmas: Lemmas~\ref{l S} and~\ref{trans on handle} will be used  for the proof of transitivity both in case
of $g=2$ and in general case. Lemma~\ref{claim2} is specific for genus 2, its generalization requires more work for general genus.

\begin{definition}[{{\it Double $\SS$-move}}]
\label{S}
Under the conditions of definition of a handle twist (Definition~\ref{Dehn}),
a  {\it double $\SS$-move} in $\h$ is the move switching the curves $a$ and $b$. 

\end{definition}

\begin{lemma}
\label{l S}
Double $\SS$-move is a morphism of $\DP_{g,0}$.

\end{lemma}

\begin{proof}
Any double $\SS$-move is a  composition of 3 handle twists, see Fig.~\ref{double s}.


\end{proof}

\begin{figure}[!h]
\begin{center}
\psfrag{a}{\scriptsize $a$}
\psfrag{b}{\scriptsize $b$}
\psfrag{a1}{\scriptsize $a'$}
\psfrag{b1}{\scriptsize $b'$}
\psfrag{a2}{\scriptsize $a''$}
\psfrag{b2}{\scriptsize $b''$}
\psfrag{a3}{\scriptsize $a'''$}
\psfrag{b3}{\scriptsize $b'''$}
\epsfig{file=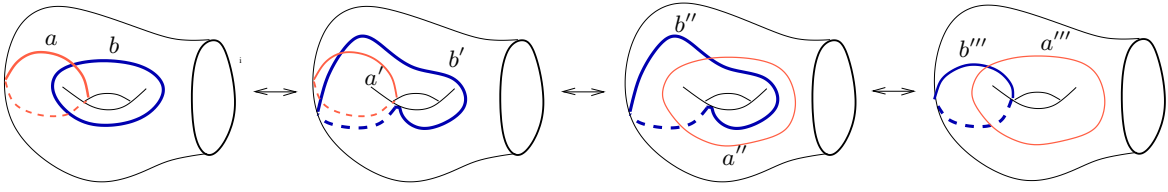,width=0.998\linewidth}
\caption{Double $\SS$-move as a composition of three handle twists.} 
\label{double s}
\end{center}
\end{figure}

\begin{lemma}
\label{homol->homot}
Let $\h$ be a handle and $c\in \h$ be a curve, $h(c)\ne 0$.
Then the homotopy class of $c$ is determined by the homology class $h(c)$. 

\end{lemma}

\begin{proof}
Attach a disk $D$ to the boundary $\partial \h$ so that we obtain a closed torus $\mathfrak t=\h\cup D$.
The statement of the lemma holds for a closed torus, so the homotopy class of $c$ is determined by $h(c)$ 
together with the relative position of the curve $c$ and the hole $\partial \h$ in the closed torus $\mathfrak t$. 
However, $\mathfrak t\setminus c$ is connected so the relative position of $c$ and $\partial \h$ is unique up to isotopy of $\mathfrak t$,
which implies that the homotopy class of $c$ is determined by $h(c)$. 

\end{proof}

\begin{lemma}
\label{trans on handle}
Let $(P_a,P_b)$ and $(P_a',P_b')$ be two standard double pants decompositions containing the same handle $\h$.
Then $(P_a,P_b)|_{\h}$ may be transformed to  $(P_a',P_b')|_{\h}$ 
by a sequence  of handle twists in $\h$
(where $(P_1,P_2)|_\h$ is a restriction of the double pants decomposition to the handle $\h$).

\end{lemma}

\begin{proof}  
Let $a,b,a',b'$ be the curves of $P_a,P_b, P_a',P_b'$ contained in $\h$. We need to find a composition $\overline \psi$
of handle twists in $\h$ such that $\overline \psi(a)=a'$,  $\overline \psi(b)=b'$.

First, suppose that $a=a'$. 
Since $\langle h(a), h(b) \rangle = \langle h(a), h(b') \rangle$, one has $$h(b')=h(b)+l_ah(a)$$ for $l_a\in \Z$.
Applying $l_a$-tuple  twist along $a$ to the curve $b$ we obtain a curve $c$ such that $h(c)=h(b')$.
In view of Lemma~\ref{homol->homot} this implies that $b'=c$. So, in the case $a=a'$, $\overline \psi$ is a composition of flips along 
$a$.

Next, suppose that $a'=b$. Then a double $\SS$-move interchanging $a$ with $b$ reduces the question to the previous case.

Now, suppose that $a'\ne a,b$. Then 
we have
$$h(a')=l_a h(a)+l_b h(b),$$ where $l_a,l_b\in \Z$ are coprime. 
In view of Lemma~\ref{homol->homot}, we only need to find a sequence $\overline \psi$
of morphisms of $\DP_{g,0}$ taking $a$ to any curve $x\in S'$  such that $h(x)=l_a h(a)+l_b h(b)$.
Since $l_a$  and $l_b$ are coprime and a handle twist transforms $(h(a),h(b))$ into either $(h(a)\pm h(b),h(b))$ or 
$(h(a),h(b)\pm h(a))$,
this sequence of handle twists do exists.

\end{proof}


\begin{lemma}
\label{claim2}
Let $(P_a,P_b)$ be a standard double pants decomposition of $S_{2,0}$.
 Let  $\varphi=\varphi_k\circ\dots\circ\varphi_1$ be a sequence of flips of $P_a$ such that $\varphi(P_a)$ is a standard decomposition. 
Then there exists a morphism $\eta$ of $\ADP_{2,0}$ and a pants decomposition $P_b'$ of $S$ such that $\eta((P_a,P_b))$ is a 
standard double pants decomposition and $\eta((P_a,P_b))=(\varphi(P_a),P_b')$. 

\end{lemma}

\begin{proof}
Recall from  Example~\ref{S_2-orbit} that for each two pants decompositions $P_1$ and $P_2$ connected by a sequence of flips
there exists a sequence of flips connecting these pants decompositions and such that each flip in this sequence 
changes the type of pants decomposition from ``self-folded'' 
into ``non-self-folded'' or back (in the other word takes a standard  pants decomposition to a non-standard one and back). 
This implies that it is sufficient to show the lemma for compositions 
$\varphi=\varphi_2\circ\varphi_1$ of two flips.

If $\varphi_2$ changes the same curve as $\varphi_1$ does, then $\varphi$ is a twist and the required composition 
$\eta$ is shown in Fig.~\ref{comp-for-twist}.

\begin{figure}[!h]
\begin{center}
\psfrag{e1}{\scriptsize $\eta_1$}
\psfrag{e2}{\scriptsize $\eta_3\circ\eta_2$}
\psfrag{e3}{\scriptsize $\eta_5\circ\eta_4$}
\psfrag{e4}{\scriptsize $\eta_6$}
\psfrag{f}{\scriptsize $f_a,f_b$}
\psfrag{S}{\scriptsize $\SS$-move}
\psfrag{f1}{\scriptsize $\varphi_1$}
\psfrag{f2}{\scriptsize $\varphi_2$}
\psfrag{a}{\small (a)}
\psfrag{b}{\small (b)}
\epsfig{file=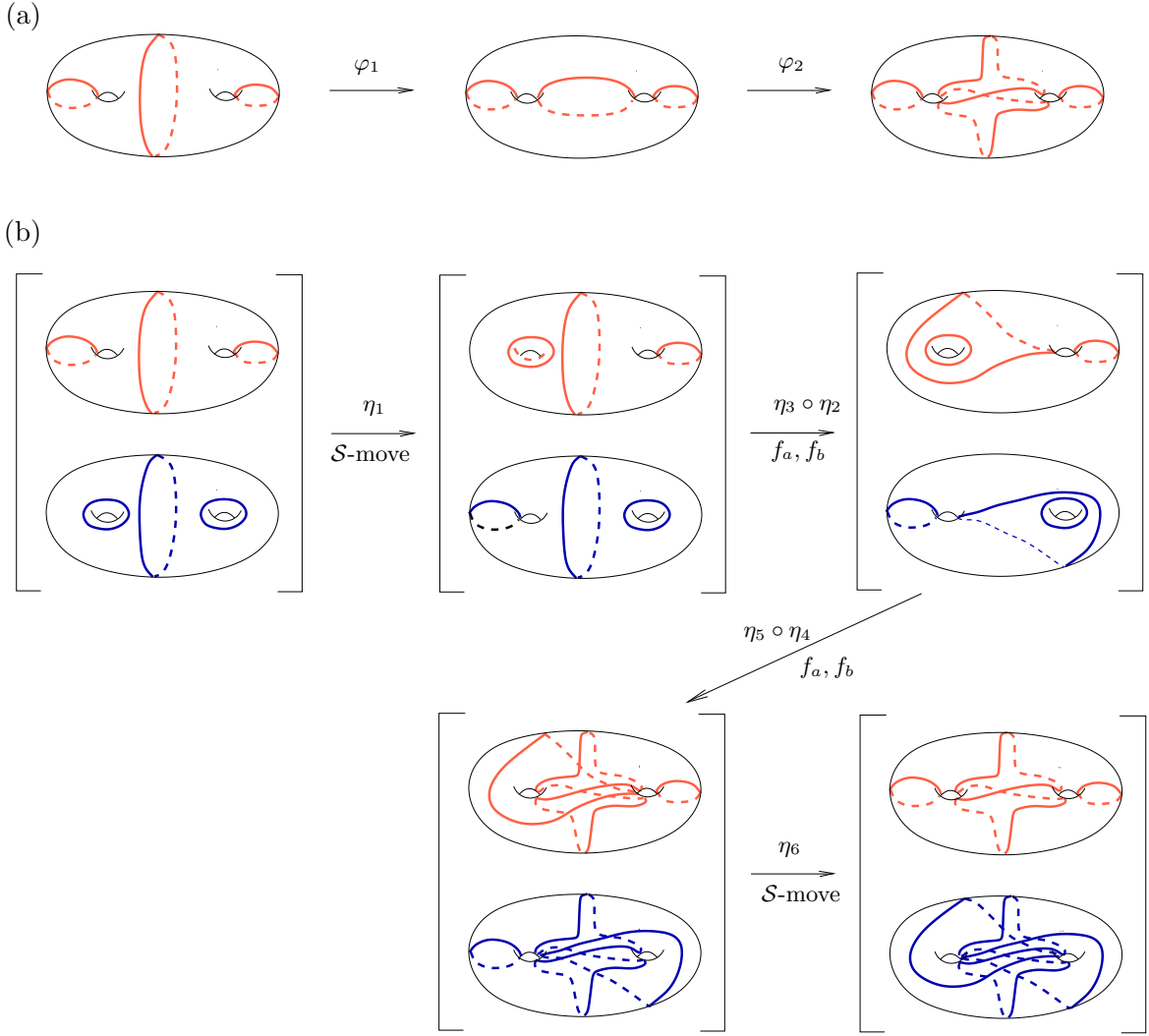,width=0.998\linewidth}
\caption{(a) Twist $\varphi(P_a)$: $\varphi=\varphi_2\circ\varphi_1$; (b) Composition $\eta(P_a,P_b)$ for the twist $\varphi$.} 
\label{comp-for-twist}
\end{center}
\end{figure}

If $\varphi_1$ and $\varphi_2$ change different curves then (modulo twists) $\varphi$  looks like in Fig.~\ref{comp-for-dflip}(a)
and the required composition 
$\eta$ is shown in Fig.~\ref{comp-for-dflip}(b).

\begin{figure}[!h]
\begin{center}
\psfrag{e1}{\scriptsize $\eta_1$}
\psfrag{e2}{\scriptsize $\eta_2$}
\psfrag{f1}{\scriptsize $\varphi_1$}
\psfrag{f2}{\scriptsize $\varphi_2$}
\psfrag{a}{\small (a)}
\psfrag{b}{\small (b)}
\epsfig{file=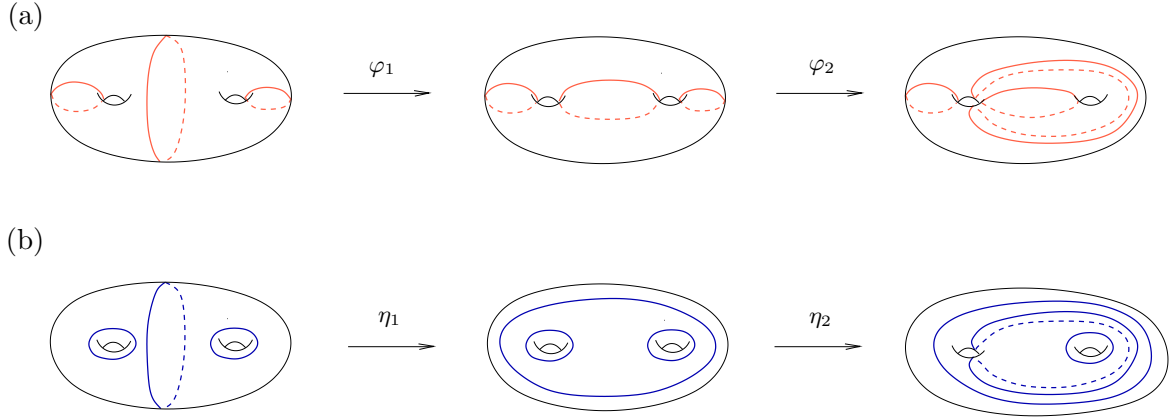,width=0.998\linewidth}
\caption{(a) Composition $\varphi(P_a)$ of two flips  $\varphi=\varphi_2\circ\varphi_1$;  
(b) Composition $\eta=\eta((P_a,P_b))$ for $\varphi$: $\eta(P_a)=\varphi(P_a)$, $\eta(P_b)$ is shown.} 
\label{comp-for-dflip}
\end{center}
\end{figure}

This completes the proof of Lemma~\ref{claim2}. 

\end{proof}

\begin{theorem}
\label{transitivnost na S_2}
Morphisms of $\ADP_{2,0}$ act transitively on the objects of $\ADP_{2,0}$.

\end{theorem}

\begin{proof}
By Definition~\ref{def-double}, the objects of $\ADP_{2,0}$ are admissible double pants decompositions, 
so, it is sufficient to prove the transitivity on the set of standard pants decompositions.  

Let $(P_a,P_b)$ and $(P_a',P_b')$ be two standard double pants decompositions. 
If the principle curve of  $(P_a,P_b)$ coincides with one of $(P_a',P_b')$ then  Lemma~\ref{trans on handle} implies that
$(P_a,P_b)$ is $\DP$-equivalent to $(P_a',P_b')$. 

Suppose that the principle curves of $(P_a,P_b)$ and $(P_a',P_b')$ are different.
It is left to show that there exists a sequence $\eta=\eta_n\circ\dots\circ \eta_1$ of morphisms of $\ADP_{2,0}$ such that 
$\eta((P_a,P_b)=(P_a',P_b'')$ and $P_b''$ is arbitrary pants decomposition turning the pair $(P_a',P_b'')$ into a standard
pants decomposition.

By Theorem~\ref{teacher}, there exists a sequence $\psi=\psi_k\circ\dots\circ \psi_1$ 
of flips and $\SS$-moves taking $P_a$ to $P_a'$.   
By definition, in case of $g=2$ an $\SS$-move is applicable only to standard double pants decompositions. 
This implies that the sequence $\psi$ is a composition of several   subsequences of two types:
\begin{itemize}
\item 
subsequences of flips, each subsequence takes a standard pants decomposition to another standard one;
\item 
$\SS$-moves.
\end{itemize}
By Lemma~\ref{claim2}, a subsequence of the first type may be extended to a morphism of $\ADP_{2,0}$ 
taking a standard pant decomposition to 
another a standard one. By Lemma~\ref{trans on handle} any $\SS$-move of the component $P_a$ may be realized as a sequence of morphisms 
of $\ADP_{2,0}$ taking a standard double pants decomposition to a standard one.
This implies that $\psi$ may be extended to a morphism of $\ADP_{2,0}$ and the theorem is proved.

\end{proof}

%
%

\section{Transitivity of morphisms in case of surfaces without marked points}
\label{g>2}
In this section we adjust the proof of Theorem~\ref{transitivnost na S_2} to the case of higher genus.

\subsection{Preparatory lemmas}

\begin{lemma}
\label{s in stand}
Let $P$ be a pants decomposition and $\h$ be a handle cut out by some $c\in P$. 
Let $\varphi$ be an $\SS$-move  in  $\h$. 
Then $\varphi=(\overline \varphi_1)^{-1}\circ \varphi_2 \circ \overline \varphi_1$ 
where $\overline \varphi_1$ 
is a sequence of flips preserving $\h$ and $\varphi_2$ is an $\SS$-move of a standard pants decomposition.

\end{lemma}

\begin{proof}
Consider the surface $S\setminus \h$.
By Theorem~\ref{teacher} it is possible to transform any pants decomposition of $S\setminus \h$ to a standard one 
(clearly it may be done using flips only: one may use zipped flips with respect to any zipper system compatible with $P$). 
This defines the sequence $\overline \varphi_1$. 
Then we apply $\SS$-move in
the handle $\h$ and apply  $\varphi_1^{-1}$ to bring the pants decomposition of  $S\setminus \h$ into initial position. 

\end{proof}

In the proof of Theorem~\ref{transitivnost na S_2} we used the fact that in case of $g=2$  $\SS$-moves are defined for standard 
decompositions only. This does not hold for $g>2$. However, in view of Lemma~\ref{s in stand} the following holds.

\begin{lemma}
\label{s standard}
Let $P_a$ and $P_a'$ be standard pants decompositions and there exists a sequence $\varphi=\varphi_k\circ\dots\circ\varphi_1$ of flips 
and $\SS$-moves such that $\varphi(P_a)=P_a'$. Then it is possible to choose $\varphi$ in such a way that all $\SS$-moves
in $\varphi$ are applied to standard decompositions.

\end{lemma}

\begin{lemma}
\label{g=0}
Let $S_{0,g}$ be a sphere with $g$ holes. Then any pants decomposition of  $S_{0,g}$  may be transformed to any other 
pants decomposition of $S_{0,g}$  by a sequence of flips.

\end{lemma}

\begin{proof}
The statement follows immediately from Theorem~\ref{teacher} since   $\SS$-moves could not be applied to $S_{0,g}$.

\end{proof}

\subsection{Transitivity in case of the same zipper system}

\begin{definition}[{{\it  $(P_a,P_b)$ compatible with $Z$}}]
A double pants decomposition $(P_a,P_b)$ is {\it compatible} with a zipper system $Z$ if $P_a$ is compatible with $Z$ and
 $P_b$ is compatible with $Z$.

\end{definition}

In this section we will prove that if  $(P_a,P_b)$  an  $(P_a',P_b')$ are two admissible double pants decompositions 
compatible with the same zipper system then  $(P_a,P_b)$ is $\DP$-equivalent to  $(P_a',P_b')$.

\begin{definition}[{{\it Strictly standard decomposition}}]
An admissible double pants decomposition $(P_a,P_b)$ is {\it strictly standard} if it is standard and 
$c\in \{P_a\cup P_b\}\setminus \{P_a\cap P_b\}$ if and only if $c$ is contained inside some handle cut out by a principle curve 
of the decomposition.

\end{definition}

\begin{example}
The double pants decomposition in Fig.~\ref{standard double_pant} is standard but not strictly standard.
\end{example}

\begin{remark}
Although we will not need this below, one can mention that
strictly standard decompositions can be also characterized by any of the following equivalent minimal properties:
\begin{itemize}
\item double pants decompositions with minimal possible number of distinct curves (i.e. with $4g-3$ curves for $S_{g,0}$);
\item double pants decompositions with minimal possible number of intersections of curves (i.e. with $g$ intersections).
\end{itemize}


\end{remark}

To prove transitivity of morphisms of $\ADP_{g,0}$ on the objects of $\ADP_{g,0}$ 
it is sufficient to prove $\DP$-equivalence of all standard 
double pants decompositions (since the objects of $\ADP_{g,0}$ are admissible ones).
Furthermore, any standard double pants decomposition is $\DP$-equivalent to some strictly standard one in view of Lemma~\ref{g=0}.
So, it is sufficient to prove the transitivity of morphisms of $\ADP_{g,0}$ on strictly standard double pants decompositions.

To prove the transitivity of morphisms of $\ADP_{g,0}$ on strictly standard double pants decompositions we do the following:
\begin{itemize}
\item
we show that each strictly 
standard double pants decomposition is compatible with some zipper system (see Proposition~\ref{compatible zip});
\item we prove that morphisms of $\ADP_{g,0}$ act transitively on strictly standard double pants decompositions 
compatible with a given zipper system
(see Lemma~\ref{same zip}); 
\item
Finally, we show that for two different zipper systems $Z$ and $Z'$ we may find a sequence of zipper systems $Z=Z_1,Z_2\dots,Z_k=Z'$,
in which $Z_i$ differs from $Z_{i+1}$ by a twist along some curve $c$, $|c\cap Z_i|=2$. We show (Lemma~\ref{l twist}) that 
in this case morphisms of $\ADP_{g,0}$ are sufficient to change $Z_i$ to $Z_{i+1}$.

\end{itemize}

\begin{prop}
\label{compatible zip}
For any strictly standard double pants decomposition $(P_a,P_b)$ there exists a zipper system $Z$ compatible with  $(P_a,P_b)$.

\end{prop}

\begin{proof}
An intersection of the required zipper system $Z=\{ z_0,z_1,\dots,z_g\}$ with a handle looks as shown in 
Fig.~\ref{Z for DP}.(a): if $a_i$ and $b_i$ are curves 
of $P_a$ and $P_b$ contained in a given handle $\h_i$, $i=1,\dots,g$, than $\h_i$ contains a zipper $z_i$ such that 
$|z_i\cap a_1|=1$ and $|z_i\cap b_1|=1$. The curve $z_0$ visits each of the handles and goes in $\h_i$ along $z_i$.

The condition that  $(P_a,P_b)$ is strictly standard leads to the existence of appropriate $z_0$ outside of handles: to show this
we build a {\it  dual graph} of  $(P_a,P_b)$ substituting each pair of pants by an $Y$-shaped figure and each handle by a point
(see Fig.~\ref{Z for DP}.(b)). Since  $(P_a,P_b)$ is strictly standard we obtain a tree. We consider this tree as a graph drawn on a 
plane $\Pi$.  Denote by $\bar z_0\in \Pi$ the curve doing around the tree. 
Then $z_0$ is built as any curve on $S$ which 
projects to $\bar z_0$ (more precisely, $\bar z_0$ determines the order in which $z_0$
visits the handles of  $(P_a,P_b)$).  

\end{proof}

\begin{figure}[!h]
\begin{center}
\psfrag{a}{\scriptsize $a_i$}
\psfrag{b}{\scriptsize $b_i$}
\psfrag{aa}{\small (a)}
\psfrag{bb}{\small (b)}
\psfrag{z}{\scriptsize $z_i$}
\psfrag{z0}{\scriptsize $z_0$}
\psfrag{h}{\scriptsize $\h$}
\epsfig{file=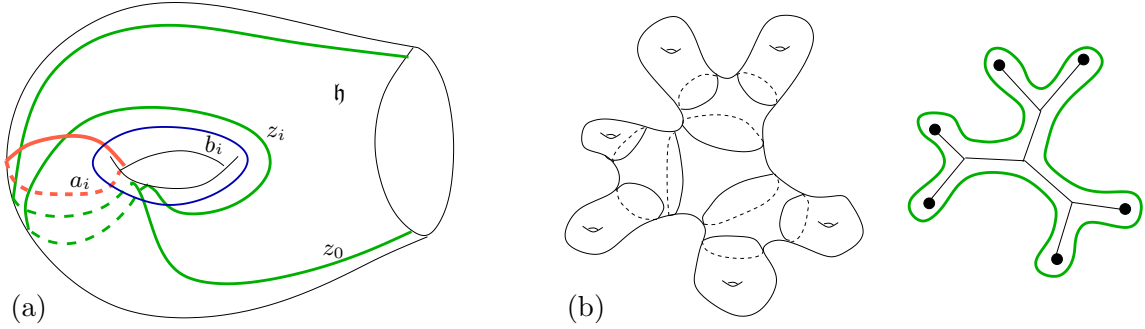,width=0.98\linewidth}
\caption{A zipper system for a given standard double pants decomposition: (a) behavior in a handle; (b) outside of the handles: 
a dual graph of a strictly standard pants decomposition.} 
\label{Z for DP}
\end{center}
\end{figure}

\begin{prop}
Let $(P_a,P_b)$ be a standard double pants decomposition and $Z$ be a zipper system compatible with $(P_a,P_b)$.
Then 
\begin{itemize}
\item
each of the handles of $(P_a,P_b)$ contains exactly one curve of $Z$, 
\item
if $z_0\in Z$ does not belong entirely to any handle then $z_0$  visits each of $g$ handles exactly once.

\end{itemize}
\end{prop}

\begin{proof}
Let $c$ be a principle curve separating a handle $\h$ in $(P_a,P_b)$. Then each curve on $S$ intersects $c$ even number of times.
By definition of a zipper system compatible with a pants decomposition, $c$ is intersected by exactly one of the curves $z_i$.
To prove that this curve is the same for all handles in $(P_a,P_b)$,
notice that a pairs of pants dissected along a connected curve does not turn into a union of simply-connected components,
which implies that there is a curve $z_j\in Z$, $z_j\ne z_i$  which intersects the handle $\h$. Since $z_j\cap c=\emptyset$, 
$z_j$ is contained in $\h$. Similarly, each of the other handles contains entirely some of the curves from $Z$.
No of these curves may coincide with $z_i$ (since $z_i$ intersects $\h$). Hence, $Z$ consists of $g$ curves
contained in the handles of $(P_a,P_b)$ and one additional curve $z_i$. 
 This proves the first statement of the proposition. 
The second statement follows from the fact that each principle curve in $(P_a,P_b)$ should be intersected by some of $z_i$.  

\end{proof}

\begin{definition}[{{\it Principle zipper, cyclic order}}]
Let $(P_a,P_b)$ be a standard double pants decomposition and $Z=\{ z_0,z_1,\dots,z_g\}$ 
be a zipper system compatible with $(P_a,P_b)$. Suppose that $z_0$ is the curve visiting all handles of $(P_a,P_b)$.
Then we call $z_0$ a {\it principle zipper}. 

A {\it cyclic order} of $Z$ is $[z_1,z_2,\dots,z_g]$ if an orientation of $z_0$ goes from $\h_i$ to $\h_{i+1}$,
where $\h_i$ is the handle containing $z_i$ and $i$ is considered modulo $g+1$ (more precisely, $Z$ decomposes $S$ into two
$(g+1)$-holed spheres $S_+$ and $S_-$, so that we may choose a positive orientation of $z_0$ as one which goes in positive direction
around $S_+$; for a definition of a cyclic order we choose the positive orientation of $z_0$). 

Given two standard double pants decompositions $DP$ and $DP'$ compatible with the same zipper system $Z$.
We say that $DP$ has {\it the same cyclic order} if the cyclic order for $DP$ determined by a positive orientation of $z_0$
coincides with the cyclic order for $DP'$ determined by the same orientation of $z_0$.

\end{definition}

\begin{remark}
The definition of cyclic order depends on the choice of $S_+$ among two subsurfaces. However, the definition of 
{\it the same cyclic order} in two standard double pants decompositions is independent of this choice
(provided that the choice of $S_+$ is the same for both decompositions).

\end{remark}

\begin{prop}
\label{same principle zip}
Let $(P_a,P_b)$ and $(P_a',P_b')$  be two standard pants decompositions compatible with the same zipper system 
$Z=\{ z_0,z_1,\dots,z_g\}$. Suppose that $z_0$ is the principle zipper of $Z$ both for   $(P_a,P_b)$ and $(P_a',P_b')$ 
and the cyclic order of $Z$ is $[z_1,z_2,\dots,z_n]$  both for   $(P_a,P_b)$ and $(P_a',P_b')$.

Then the set of principle curves of $(P_a,P_b)$ coincides with the set of principle curves of $(P_a',P_b')$.

\end{prop}

\begin{proof}
 Let $c_1,\dots,c_g$ be the principle curves of $(P_a,P_b)$ and  $c_1',\dots,c_g'$ be the principle curves of $(P_a',P_b')$.  
Since the cyclic order of $Z$ is the same for both double pants decompositions, we may assume that $z_0\cap c_i=z_0\cap c_i'$.
Let $S^+$ and $S^-$ be  the connected components of $S\setminus Z$.
Then $c_i\cap S^+$ separates from $S^+$ an annulus containing $z_i$ as a boundary component. 
Clearly, the same holds for $c_i\cap S^-$ as well as for $c_i'\cap S^+$ and $c_i'\cap S^-$. 
This implies that $c_i$ is homotopy equivalent to $c_i'$ (more precisely, there exists an isotopy of $c_i$ to $c_i'$ with the fixed points
$c_i\cap z_0=c_i'\cap z_0$).

\end{proof}

%

\begin{cor}
\label{cor}
In assumptions of Proposition~\ref{same principle zip}, $(P_a,P_b)$ may be transformed to $(P_a',P_b')$ 
by morphisms of $\ADP_{g,0}$ preserving the principle curves of the standard pants decomposition.

\end{cor}

\begin{proof}
This follows from  Proposition~\ref{same principle zip}, Lemma~\ref{trans on handle} and Lemma~\ref{g=0}.
\end{proof}

Corollary~\ref{cor} implies that a standard pants decomposition is determined 
(modulo action of morphisms of $\ADP_{g,0}$)  by the set of principle curves. Thus, it makes sense to consider a set of principle curves
itself, independently of remaining curves in the corresponding double pants decomposition.

\begin{definition}
A zipper system $Z$ is {\it compatible with a set of principle curves} if $Z=\{ z_0,z_1,\dots,z_g\}$ where
$z_0$ visits each handle exactly once and each of the handles contain exactly one of $z_i$, $i=1,\dots,g$. 

\end{definition}

In other words, $Z$ is compatible with a set of principle curves if and only if it is compatible with a standard pants 
decomposition containing this set of principle curves.

\begin{prop}
\label{transposition}
Let $(P_a,P_b)$ and $(P_a',P_b')$  be two standard pants decompositions compatible with the same zipper system 
$Z=\{ z_0,z_1,\dots,z_g\}$. Suppose that $z_0$ is the principle zipper of $Z$ both for   $(P_a,P_b)$ and $(P_a',P_b')$. 
Then $(P_a,P_b)$ is $\DP$-equivalent to $(P_a',P_b')$. 

\end{prop}

\begin{proof}
By  Proposition~\ref{same principle zip} together with Corollary~\ref{cor} 
the proposition is trivial unless the cyclic order of $Z$ is 
different for the cases of $(P_a,P_b)$ and $(P_a',P_b')$. 
It is shown in Fig.~\ref{cyclic order} that the transposition of two neighboring zippers
$z_i$ and $z_{i+1}$ in the cyclic order of $Z$ may be realized by morphisms of $\ADP_{g,0}$.

In more details,  in  Fig.~\ref{cyclic order}, upper row, left, we show (a part of) a zipper system $Z$ compatible with a set of 
principle curves. In Fig.~\ref{cyclic order}, lower row, left we show a strictly standard double pants decomposition 
$(P _a^{(0)},P_b^{(0)})$ compatible with $Z$ and having principle curves as in the figure above. 
Applying  handle-twists in  curves of $P_a$ (one for each handle) we
obtain the decomposition  $(P _a^{(1)},P_b^{(1)})$ 
(we draw only the front, `` visible'' part of the decomposition, the non-visible part completely repeats it). Applying two flips 
(one for $P_a$ and one for $P_b$) we obtain
 a double pants decomposition $(P_a^{(2)},P_b^{(2)})$, and then after one more flip we obtain  a standard double pants decomposition
$(P_a^{(3)},P_b^{(3)})$. Choosing appropriate curves in the handles of $(P_a^{(3)},P_b^{(3)})$ we turn 
it into a standard double pants decomposition  $(P_a^{(4)},P_b^{(4)})$ compatible with $Z$
 (this is done in the same way as the transformation from  $(P _a^{(0)},P_b^{(0)})$ to  $(P _a^{(1)},P_b^{(1)})$, i.e. 
using handle twists in each of the handles).  In  Fig.~\ref{cyclic order}, upper row, right, we show the zipper system $Z$ together
with the principle curves of  $(P_a^{(4)},P_b^{(4)})$.
Notice that the cyclic order in $Z$ is changed: 
in the initial decomposition (an orientation of) the principle zipper $z_0$ visits first 
the handle containing $z_1$ and then the handle containing $z_2$, while in the final decomposition the same orientation of $z_0$
visits the handles in reverse order. 
So, the transposition of two adjacent (in the cyclic order) handles is realizable by the morphisms of 
$\ADP_{g,0}$.  

Since the permutation group is generated by transpositions of adjacent elements, the proposition is proved.

\end{proof}

\begin{figure}[!h]
\begin{center}
\psfrag{z0}{\scriptsize $z_0$}
\psfrag{z1}{\scriptsize $z_1$}
\psfrag{z2}{\scriptsize $z_2$}
\psfrag{Z}{\small $Z$}
\psfrag{0}{\small  $(P_a^{(0)},P_b^{(0)})$}
\psfrag{1}{\small  $(P_a^{(1)},P_b^{(1)})$}
\psfrag{2}{\small  $(P_a^{(2)},P_b^{(2)})$}
\psfrag{3}{\small  $(P_a^{(3)},P_b^{(3)})$}
\epsfig{file=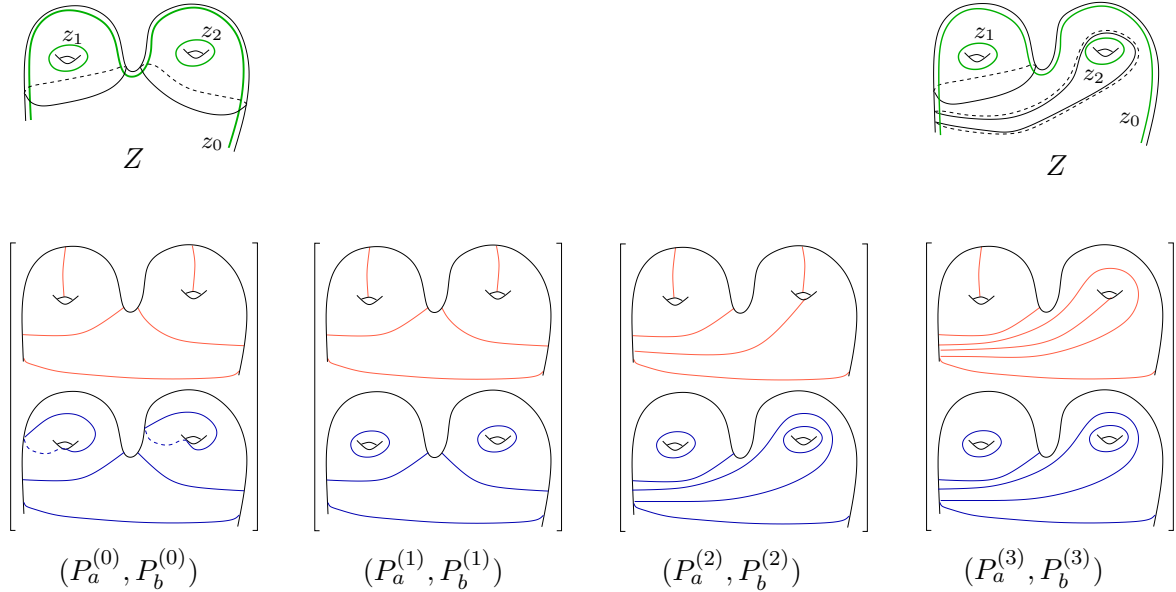,width=0.998\linewidth}
\caption{Transposition in a cyclic order is realizable by morphisms of $\ADP_{g,0}$: apply  handle-twists in each handle, 
then a flip in each pants decomposition, then a flip in $P_b$, and again, handle twists in each of two handles.} 
\label{cyclic order}
\end{center}
\end{figure}

\begin{lemma}
\label{same zip}
Let $(P_a,P_b)$ and $(P_a',P_b')$  be two standard pants decompositions compatible with the same zipper system $Z$. 
Then $(P_a,P_b)$ is $\DP$-equivalent to $(P_a',P_b')$. 

\end{lemma}

\begin{proof}

Suppose that the principle zipper of $Z$ is different for  $(P_a,P_b)$ and $(P_a',P_b')$
(otherwise there is nothing to prove in view of  Proposition~\ref{transposition}).

By  Proposition~\ref{transposition},  
the lemma is trivial if the principle zipper of $Z$ is the same for  $(P_a,P_b)$ and $(P_a',P_b')$. 
So, we only need to show that the morphisms of $\DP_{g,0}$ allow to change the principle zipper in $Z$.
We will show that it may be done by flips only.

Let $Z$ be a zipper system (see Fig.~\ref{principle zip}) compatible with a set $\bar c$ of principle curves.  
Let $(P_a,P_b)$ be a standard double pants decomposition containing this set of principle curves.
Let $\bar c'$ be another set of principle curves compatible with $Z$ and such that $z_0$ is not a principle zipper
(see  Fig.~\ref{principle zip}, down).  
We choose a standard double pants decomposition  $(P_a',P_b')$ with a set of principle curves $\bar c'$, 
see  Fig.~\ref{principle zip} (to keep the figure readable we do not draw the curves of $(P_a',P_b')$ 
decomposing $S\setminus \cap_{i=1}^g \h_g$).
To prove the lemma it is sufficient to show that $P_a$ is flip-equivalent to $P_a'$ and $P_b$ is flip-equivalent to $P_b'$.

\begin{figure}[!h]
\begin{center}
\psfrag{a}{\scriptsize $a_i$}
\psfrag{b}{\scriptsize $b_i$}
\psfrag{z}{\scriptsize $Z,\bar c$}
\psfrag{z1}{\scriptsize $Z,\bar c'$}
\psfrag{1}{\scriptsize $(P_a,P_b)$}
\psfrag{2}{\scriptsize $(P_a',P_b')$}
\psfrag{a}{\small (a)}
\psfrag{b}{\small (b)}
\epsfig{file=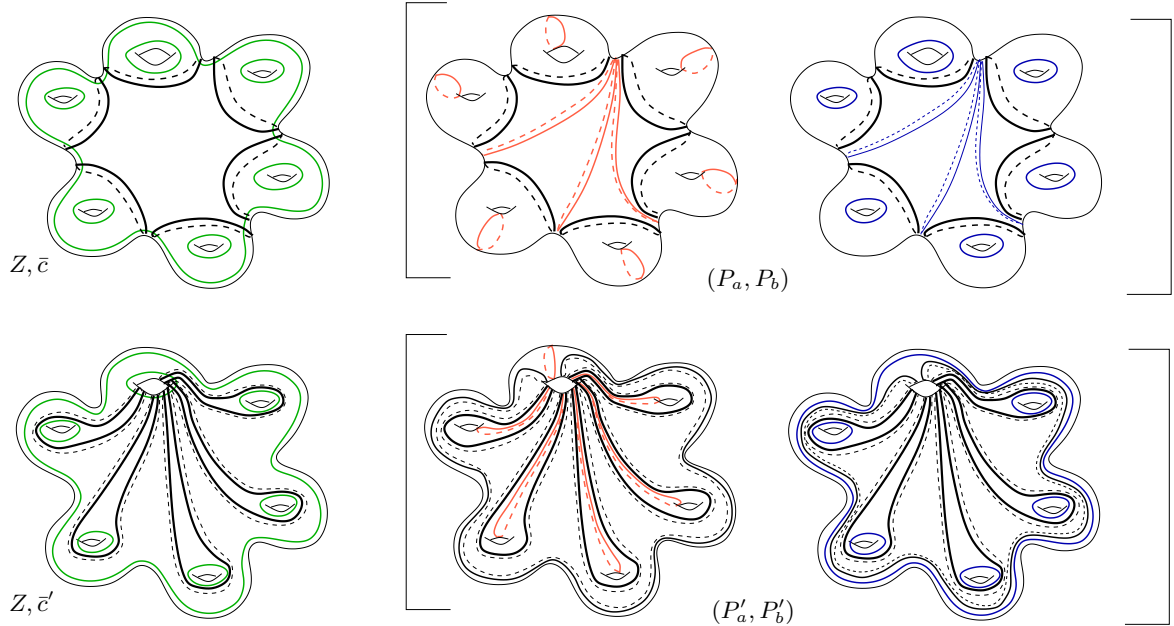,width=0.998\linewidth}
\caption{Principle zipper may be changed by morphisms of $\ADP_{g,0}$.} 
\label{principle zip}
\end{center}
\end{figure}

The fact that $P_a$ is flip-equivalent to $P_a'$ follows from Proposition~\ref{handlebody}, indeed both $P_a$ and $P_a'$ are contractible
in the inner handlebody in the embedding shown in Fig.~\ref{principle zip}.
Similarly, $P_b$ is flip-equivalent to $P_b'$, since both of them are contractible in the outer handlebody,
so the lemma is proved. 

\end{proof}

\subsection{Proof of transitivity in general case}

\begin{lemma}
\label{l twist}
Let $Z$ be a zipper system, let $\sigma$ be an involution preserving $Z$ pointwise and let $c$ be a curve satisfying $|Z\cap c|=2$,
$\sigma(c)=c$. Denote by $T_c$ a Dehn twist along $c$.
Then there exist standard pants decompositions $(P_a,P_b)$ and $(P_a',P_b')$ compatible
with $Z$ and $Z'=T_c(Z)$ respectively and $\DP$-equivalent to each other.

\end{lemma}

\begin{proof}
We will construct the decomposition $(P_a,P_b)$ so that each curve in it is preserved by $\sigma$.
For simplicity we will draw only one of two connected components $S_+$ and $S_-$ of $S\setminus Z$, 
this is sufficient in view of symmetry. We will draw $S_+$ as a disk with $g$ holes and an outer boundary coming from the principle 
zipper.

Since  $|Z\cap c|=2$,
the curve $c$ either intersects twice the same curve  $z_0\in Z$ or have  single intersections with two distinct curves $z_1,z_2\in Z$.
Consider these two cases.

Suppose that $|c\cap z_0|=2$ and $z_0\in Z$. We will build $(P_a,P_b)$ so that $z_0$ will be  a principle zipper.
Since $\sigma(c)=c$, the curve $c$ decomposes $S_+$ into two parts, 
as in Fig.~\ref{invol}(a) (each part containing at least one hole, otherwise the intersection $c\cap Z$ is not essential). 
We build a strictly standard double pants decomposition $(P_a,P_b)$ containing $c$, 
compatible with $Z$ and such  that $z_0$ is a principle zipper (see Fig.~\ref{invol}(b) for the construction: 
for each of the holes we draw a segment of the principle curves encircling this hole).
We take the second copy of the same pattern as $S_-$ and glue to $S_+$ along the boundary. 
Clearly, we obtain a set of principle curves for a standard double pants decomposition  $(P_a,P_b)$ compatible with $Z$.  
It is easy to see that the same decomposition $(P_a,P_b)$ is also compatible with $Z'=T_c(Z)$.

\begin{figure}[!h]
\begin{center}
\psfrag{c}{\scriptsize $c$}
\psfrag{s}{\small $\sigma$}
\psfrag{a}{\small (a)}
\psfrag{b}{\small (b)}
\epsfig{file=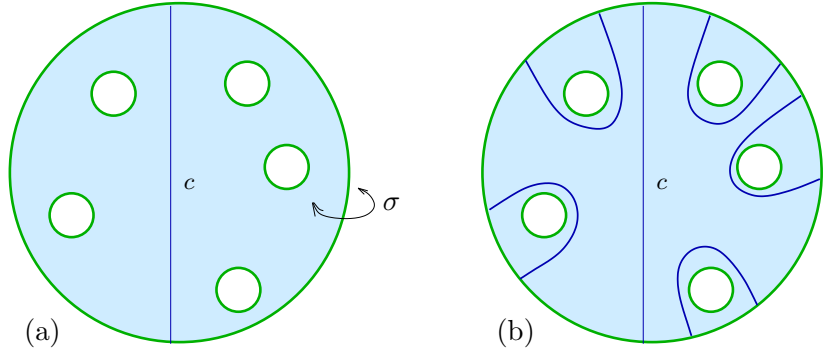,width=0.6998\linewidth}
\caption{ Case $|c\cap z_0|=2$: (a) $c$ decomposes $S_+$; (b) strictly standard double pants decomposition
(only principle curves and $c$ are shown).} 
\label{invol}
\end{center}
\end{figure}

Suppose that $|c\cap z_1|=|c\cap z_2|=1$, $z_1,z_2\in Z$  (see  Fig.~\ref{invol1}(a)). 
Choose $(P_a,P_b)$ so that $z_1$ and $z_2$ are not principle zippers and 
$z_1$ and $z_2$ are two neighboring curves in the cyclic order (see  Fig.~\ref{invol1}(b)). 
Fig.~\ref{obtain twist} contains a sequence of morphisms of $\DP_{g,0}$ taking $(P_a,P_b)$ to
a standard double pants decomposition compatible with $Z'=T_c(Z)$.

\begin{figure}[!h]
\begin{center}
\psfrag{z1}{\scriptsize $z1$}
\psfrag{z2}{\scriptsize $z2$}
\psfrag{c}{\scriptsize $c$}
\psfrag{s}{\small $\sigma$}
\psfrag{a}{\small (a)}
\psfrag{b}{\small (b)}
\epsfig{file=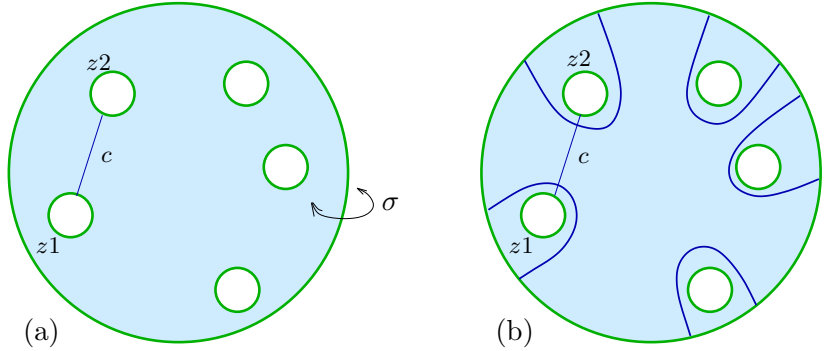,width=0.6998\linewidth}
\caption{ Case $|c\cap z_1|=|c\cap z_2|=1$: (a) $c$ in $S_+$; (b) strictly standard double pants decomposition
(only principle curves and $c$ are shown).} 
\label{invol1}
\end{center}
\end{figure}


\end{proof}

\begin{figure}[!h]
\begin{center}
\psfrag{c}{\scriptsize $c$}
\psfrag{S}{\small $S$}
\epsfig{file=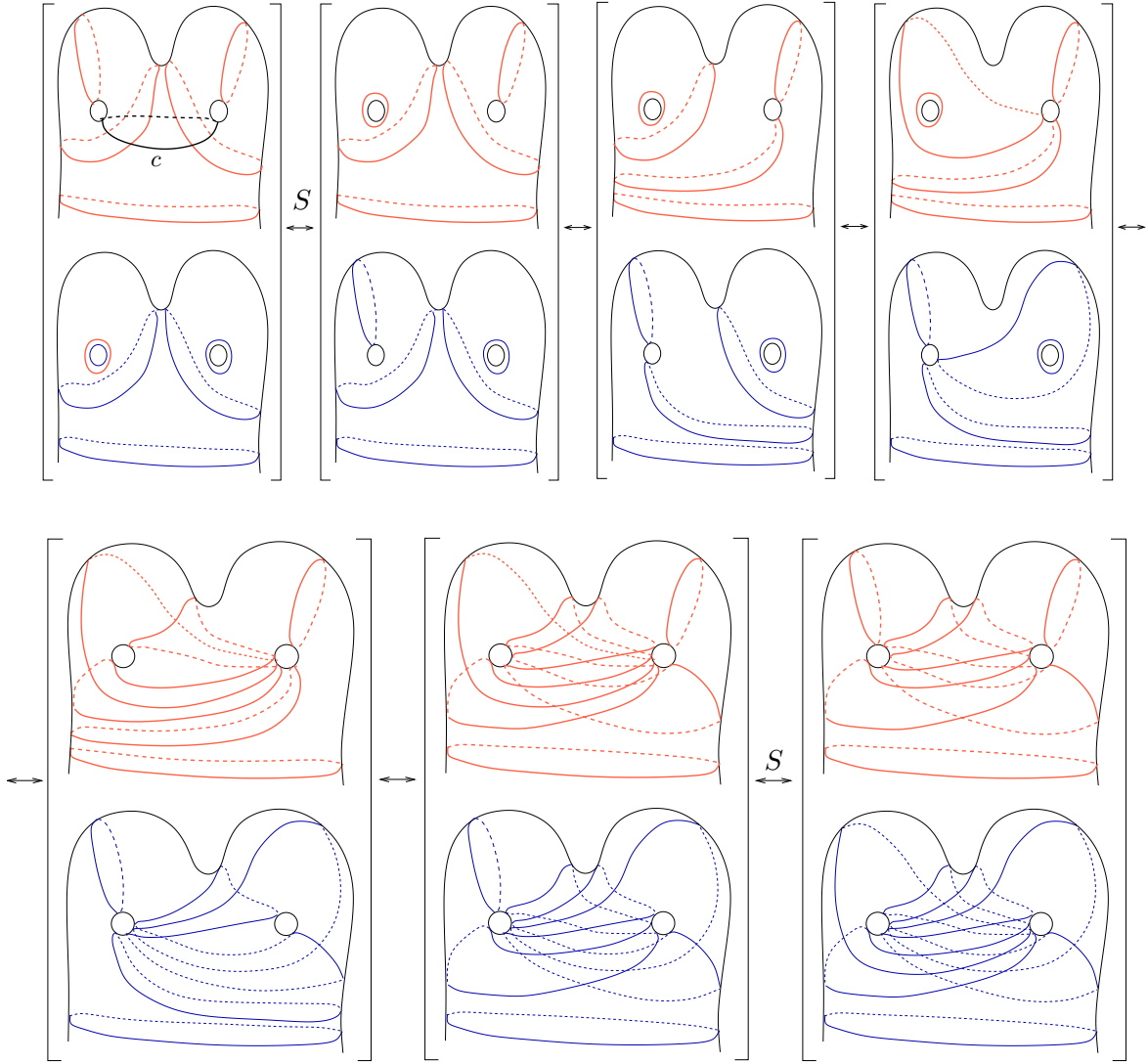,width=0.998\linewidth}
\caption{Twist $T_c$ along $c$ as a composition of elementary morphisms of $\DP_{g,0}$: the steps labeled by ``$S$'' are double $\SS$-moves, 
other steps are compositions of two flips, one in $P_a$ another in $P_b$. 
The final figure coincides with the initial one twisted around $c$.} 
\label{obtain twist}
\end{center}
\end{figure}

Applying Lemma~\ref{l twist} we obtain the following corollary. 

\begin{cor}
\label{cor l twist}
In assumptions of  Lemma~\ref{l twist},
the statement of Lemma~\ref{l twist} holds for $Z$ and $Z''=T_c^m(Z)$ for any positive integer degree $m$.

\end{cor}

\begin{definition}[{{\it $\DP$-equivalent standard pants decompositions}}]
A standard pants decomposition $P_a$ is {\it $\DP$-equivalent} to a standard pants decomposition $P_a'$ if
there exist standard pants decompositions $P_b$ and $P_b'$ such that $(P_a,P_b)$ and $(P_a',P_b')$ are  standard and 
$\DP$-equivalent to each other.

\end{definition}

\begin{theorem}
\label{trans}
Let $S=S_{g,0}$ a surface without marked points. Then
morphisms of $\ADP_{g,0}$ act transitively on the objects of $\ADP_{g,0}$.

\end{theorem}

\begin{proof}
It is clear from the definition of admissible pants decomposition that it is sufficient to prove transitivity for standard pants
decompositions only.
By  Lemmas~\ref{trans on handle} and~\ref{g=0}  morphisms of $\DP_{g,0}$ act transitively on standard pants decompositions 
with the same principle curves. 
This implies that it is sufficient to show that any two standard pants decompositions $P_a$ and $P_a'$ are $\DP$-equivalent.

By Theorem~\ref{teacher} there exists a sequence $\{\varphi_i\}$  
 of flips and $\SS$-moves 
taking $P_a$ to $P_a'$ . In view of Lemma~\ref{s in stand} we may assume that in this sequence $\SS$-moves are applied 
only to the standard double pants decompositions. Lemma~\ref{trans on handle} treats the $\SS$-moves in standard pants decompositions,
thus, we may assume that  $\{\varphi_i\}$ consists entirely of flips of $P_a$.

If $Z$ is a zipper system compatible with  $(P_a,P_b)$ and all flips  in $\{\varphi_i\}$
are zipped flips (with respect to $Z$) then there is nothing to prove.
Our idea is to decompose  the sequence $\{\varphi_i\}$ into several subsequences 
$$\{\varphi_i\}=\{\{\varphi_i\}_1,\dots,\{\varphi_i\}_k \}$$
such that in $j$-th subsequence all flips are zipped flips with respect to the same zipper system $Z_j$. 
Denote by $P_a^{j}$ the pants decomposition obtained from $P_a$ after application of the first $j$ subsequences of flips, $P_a^0=P_a$, 
$P_a^k=P_a'$.
Clearly, $P_a^{j}$ is compatible both with $Z_j$ and $Z_{j+1}$. So, we may use zipped flips (with respect to $Z_j$) to transform   
$P_a^{j}$ to some standard pants decomposition $P_a^{j}(Z_j)$ compatible with $Z_j$.  Similarly, we may transform 
$P_a^{j}$ to some standard pants decomposition $P_a^{j}(Z_{j+1})$ compatible with $Z_{j+1}$ (see Fig.~\ref{pr-trans}).

\begin{figure}[!h]
\begin{center}
\psfrag{0}{\scriptsize $P_a=P_a^0$}
\psfrag{1}{\scriptsize $P_a^1$}
\psfrag{2}{\scriptsize $P_a^2$}
\psfrag{j1}{\scriptsize $P_a^{j-1}$}
\psfrag{j}{\scriptsize $P_a^j$}
\psfrag{k}{\scriptsize $P_a^k=P_a'$}
\psfrag{Z1}{\scriptsize $Z_1$}
\psfrag{Z2}{\scriptsize $Z_2$}
\psfrag{Z3}{\scriptsize $Z_3$}
\psfrag{Zj}{\scriptsize $Z_j$}
\psfrag{Zj1}{\scriptsize $Z_{j-1}$}
\psfrag{Zj11}{\scriptsize $Z_{j+1}$}
\psfrag{P1}{\scriptsize $P_a^1(Z_1)$}
\psfrag{P2}{\scriptsize $P_a^1(Z_2)$}
\psfrag{P3}{\scriptsize $P_a^2(Z_2)$}
\psfrag{P4}{\scriptsize $P_a^2(Z_3)$}
\psfrag{P5}{\scriptsize $P_a^{j-1}(Z_{j-1})$}
\psfrag{P6}{\scriptsize $P_a^{j-1}(Z_j)$}
\psfrag{P7}{\scriptsize $P_a^j(Z_j)$}
\psfrag{P8}{\scriptsize $P_a^j(Z_{j+1})$}
\epsfig{file=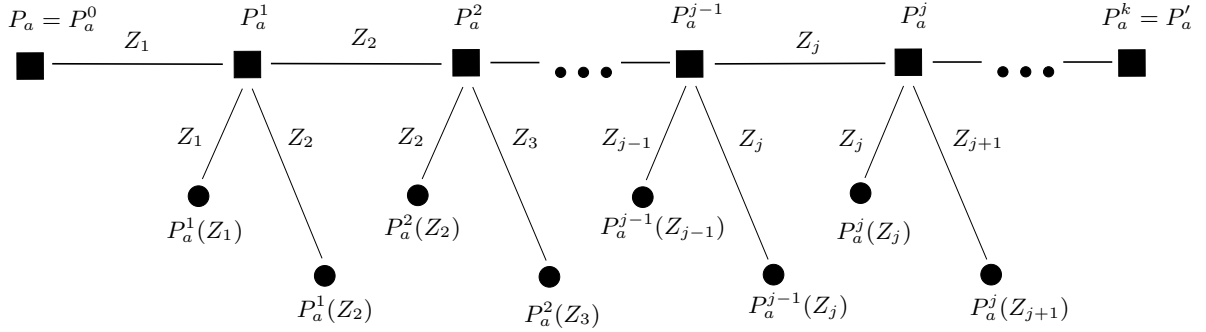,width=0.9798\linewidth}
\caption{Proof of transitivity  for $g>2$.} 
\label{pr-trans}
\end{center}
\end{figure}

Now, for proving $\DP$-equivalence of $P_a$ and $P_a'$ we only need to show two facts:

\begin{itemize}
\item[(1)]  $P_a^{j-1}(Z_j)$  is $\DP$-equivalent to $P_a^{j}(Z_j)$ for $0<j\le k$;
\item[(2)]  $P_a^j(Z)$ is $\DP$-equivalent to $P_a^j(Z_{j+1})$  for $0<j<k$.

\end{itemize}

The first of this facts follows immediately from  Lemma~\ref{same zip} since both  $P_a^{j-1}(Z_j)$ and $ P_a^{j}(Z_j)$
are compatible with the same zipper system $Z_j$.

So, we are left to prove the second fact. By Lemma~\ref{adjacent zips}, we may assume  $Z_{j+1}=T_c^{m}(Z_j)$,
where $T_c$ is a Dehn twist along a curve $c\in P_a^j$. By Remark~\ref{involution}, there exists an involution $\sigma$ 
preserving $Z$ pointwise and such that $\sigma(c_i)=c_i$ for each curve $c_i\in P_a$, so we may apply Lemma~\ref{l twist}.
By Lemma~\ref{l twist} (together with Corollary~\ref{cor l twist}) there exists a pair of $\DP$-equivalent 
double pants decompositions, one  compatible with $Z_{j}$ and another compatible with $Z_{j+1}$. In view of 
 Lemma~\ref{same zip}, this implies that $P_a^j(Z)$ is $\DP$-equivalent to $P_a^j(Z_{j+1})$.

\end{proof}




\section{Surfaces with marked points}
\label{s open}

In this section we generalize Theorem~\ref{trans} and Theorem~\ref{top} to the case of surfaces with boundary
or for surfaces with marked points.

\begin{remark}
\label{marked}
We prefer to work with surfaces with holes instead of surfaces with marked points (the latter may be obtained from the former by
contracting the boundary components).
 Since we never consider the boundary and the neighbourhood or the boundary, this makes no difference for our reasoning. 
In case of marked surfaces one need to extend the definition of a ``pair of pants'': for marked surfaces a pair of pants
is a sphere with 3 ``features'', each of the features may be either a hole or a marked point.
\end{remark}

Let $S_{g,n}$ be a surface of genus $g$ with $n$ holes. Definitions of pants decomposition and double pants decomposition
remain the same as in case of $n=0$.

\begin{definition}[{{\it  Standard double pants decomposition of an open surface}}]
A double pants decomposition $(P_a,P_b)$ of $S_{g,n}$ is {\it standard} if $P_a$ and $P_b$ contain the same set of $g$ handles
(where a {\it handle} is a surface of genus 1 with 1 hole).

A standard  double pants decomposition is {\it strictly standard} if any curve of  $(P_a,P_b)$  either belongs to both of
$P_a$ and $P_b$ or is contained in some of $g$ handles.

\end{definition}

In the same way as in case of $n=0$  we define: flips and handle twists, admissible double pants decompositions and the category
$\ADP_{g,n}$ of admissible pants decompositions. We will consider objects of $\ADP_{g,n}$  as surfaces with holes, but contracting the 
boundaries of the surface we obtain the equivalent category whose objects are surfaces with marked points.

Before proving the transitivity of morphisms on the objects of $\ADP_{g,n}$, we reprove Theorem~\ref{teacher}
for the case of open surfaces (the result is contained in~\cite{H1}, but  we present the prove here, since
the same idea works for  double pants decompositions).  
More precisely, we derive the result for open surfaces from the result for closed ones.

Given a pants decomposition of an open surface, $\SS$-moves are defined in the same way as for closed surfaces.

\begin{lemma}[A.~Hatcher,~\cite{H1}]
\label{open1}
Flips and $\SS$-moves act transitively on pants decompositions of  $S_{g,n}$.

\end{lemma}

\begin{proof}
The proof is by induction on the number of holes $n$. The base ($n=0$) is 
Theorem~\ref{teacher} for the case of closed surfaces.
Suppose that the lemma holds for $n=k$ and consider a surface $S_{g,k+1}$. 

We consider simultaneously two surfaces, $S_{g,k+1}$ and $S_{g,k}$, where the latter is thought as a copy of $S_{g,k+1}$ 
with a disk attached to the boundary of $(k+1)$-th hole. Each curve on  $S_{g,k+1}$ turns into a curve on  $S_{g,k}$
(but two distinct curves may became the same). Any pants decomposition of  $S_{g,k+1}$ turns into a pants decomposition of 
 $S_{g,k}$ containing one pair of pants less than the initial ones. More precisely, each pants decomposition  $S_{g,k+1}$ 
contains a unique pair of pants one of whose boundary components is $(k+1)$-th hole. This pair of pants disappears in  $S_{g,k}$ 
(when the hole is removed, two other boundary components turn in the same curve).
To go back from  a pants decomposition of $P_{g,k}$ of $S_{g,k}$ to a pants decomposition of  $S_{g,k+1}$  
we need only to choose one of the curves
$c\in P_{g,k}$ and attach in the place of $c$ a thin strip containing a hole.

A flip as in Fig.~\ref{hole} allow to change the curve $c\in P(S_{g,k})$ where the holed strip as attached (this flip  in the 
decomposition of $S_{g,k+1}$ does not change the decomposition of $S_{g,k}$). Applying a sequence of flips we may move the 
strip to any given curve of the pants decomposition of $S_{g,k}$.
Furthermore, for any flip or $\SS$-move in the 
decomposition of  $S_{g,k}$ we may apply similar transformation in  $S_{g,k+1}$ 
(we only need to check in advance that the holed strip is not 
attached to the curve changed by the transformation, in the latter case, first we need to change the ``stripped'' curve).
So the transitivity of flips and $\SS$-moves on pants decompositions of $S_{g,k+1}$ follows now from transitivity for  $S_{g,k}$
and a fact that flips allow as to choose the stripped curve arbitrary.

\end{proof}

\begin{figure}[!h]
\begin{center}
\epsfig{file=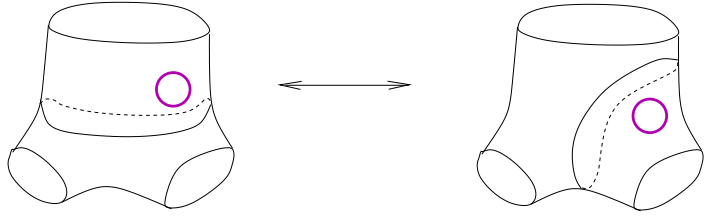,width=0.59798\linewidth}
\caption{Flip changing the curve where the holed strip is attached.} 
\label{hole}
\end{center}
\end{figure}

\begin{definition}[{{Simple double pants decomposition }}]
A double pants decomposition $(P_a,P_b)$ is {\it simple} if $|a_i\cap b_j|\le 1$ for all curves $a_i\in P_a$, $b_j\in P_b$.

\end{definition}

\begin{theorem}
\label{open}
Morphisms of $\ADP_{g,n}$ act transitively on the elements of $\ADP_{g,n}$.

\end{theorem}

\begin{proof}
The proof is by induction on the number of holes $n$. The base ($n=0$) is proved in 
Theorem~\ref{trans}.
Suppose that the theorem holds for $n=k$ and consider a surface $S_{g,k+1}$. 

Following the proof of Lemma~\ref{open1}, we consider simultaneously double pants decompositions $(P_a,P_b)$
of  $S_{g,k+1}$ and $(\widetilde P_a,\widetilde P_b)$ of  $S_{g,k}$.
Each of two pants decompositions of  $S_{g,k+1}$ differs from corresponding pants decomposition of  $S_{g,k}$ by a holed strip
attached in some of curves (so that $(k+1)$-th hole in  $S_{g,k+1}$ is lying in the intersection of two strips).
The transitivity for the case of  $S_{g,k}$ shows that flips and handle twists are sufficient to transform 
the double pants decomposition of  $S_{g,k+1}$ to one which projects to any given double pants decomposition of  $S_{g,k}$.
As it is shown in the proof of  Lemma~\ref{open1}, flips also allow to choose the curves of  $(\widetilde P_a,\widetilde P_b)$
where the holed strips are attached. 
This implies transitivity of morphisms of $\ADP_{g,n}$ 
on all  admissible double pants decomposition of  $S_{g,k+1}$ which project to simple double pants decomposition
of  $S_{g,k}$.

The reasoning above does not work for   double pants decomposition of  $S_{g,k+1}$ which do not  project to simple double pants 
decompositions of  $S_{g,k}$: indeed, in this case we may choose the double pants decomposition  $(\widetilde P_a,\widetilde P_b)$ 
of  $S_{g,k}$ and the curves $a_i$ and $b_j$ where the strips are attached, but in the case $|a_i\cap b_j|>1$ we are not able to
choose which of the intersections of the strips contains the hole. 

To adjust the proof to this case, notice that a strictly standard double pants decomposition of  $S_{g,k+1}$ projects to a strictly
standard double pants decomposition of  $S_{g,k+1}$, which is simple. This implies transitivity on strictly standard pants decompositions,
and hence, on standard ones. In view of definition of admissible double pants decomposition 
(as one which may be obtained from a standard one), we have transitivity for all admissible double pants decompositions of
 $S_{g,k+1}$. 

Thus, given the statement for $n=k$ we have proved it for $n=k+1$, hence, the theorem holds for any integer $n\ge 0$. 

\end{proof}

%
%

Theorem~\ref{open} 
completes the proof of the Main Theorem.

\section{Flip-twist groupoid and mapping class group}
\label{sec top}

In general, the elements of the flip-twist groupoid $FT$ change the topology of the double pants decomposition, 
so $FT$ is not a group. However, there are 
some elements which preserve the topology. Clearly, these elements belong to mapping class group $MCG(S)$ of the surface
(recall that a {\it mapping class group} $MCG(S)$ of a surface $S$ is a group of homotopy classes 
of self-homeomorphisms of $S$ with a composition as a group operation).
In fact, all elements of mapping class group occur to belong to $FT$.

\medskip

In this section,
we consider the curves of a double pants decomposition as {\it labeled curves}, so that a symmetry of $S$ interchanging the curves 
would not be trivial. 

A double pants decomposition $(P_a,P_b)$ contains a {\it double curve} $c$ if $c\in P_a\cap P_b$. 

\begin{lemma}
\label{ideal}
Let $S=S_{g,n}$ be a genus $g$ surface with $n\ge 0$ marked points, where $2g+n>2$.   
Let $(P_a,P_b)$ be an admissible double pants decomposition without double curves.
Then  
 $g\in MCG(S)$ fixes $(P_a,P_b)$ if and only if $g=id$. 

\end{lemma}

\begin{proof}

Suppose there exists an element $g\in MCG(S)$ such that $g((P_a,P_b))=(P_a,P_b)$, $g\ne id$.
Since $g(P_a)=P_a$, $g$ is a composition of Dehn twists along the curves contained in $P_a$. Since the curves do not intersect each other,
the Dehn twists do commute. Let $a_1\in P_a$ be any curve whose twist contributes to $g$. By assumption, $a_1\notin P_b$.
Hence, there exists a curve $b_1\in P_b$ such that $b_1\cap a_1\ne \emptyset$ 
(otherwise $P_b$ is not a maximal set of non-intersecting curves in $S$).
Then $g(b_1)\ne b_1$, so $g(P_b)\ne P_b$
which contradicts to the assumption. The contradiction implies the lemma.   

\end{proof}

An example of such a decomposition on a surface with marked points is shown in Fig.~\ref{simple_holed}

\begin{figure}[!h]
\begin{center}
\epsfig{file=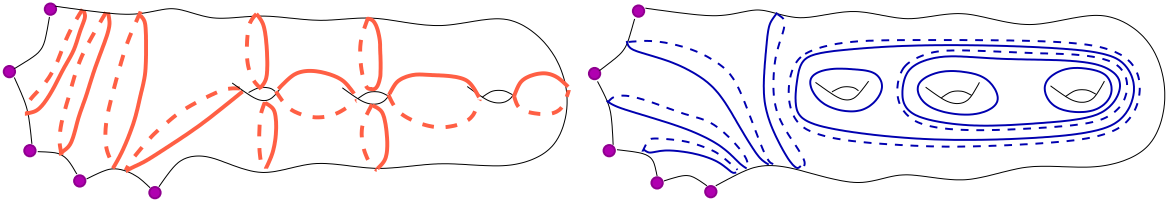,width=0.998\linewidth}
\caption{Double pants decomposition $(P_a,P_b)$ without double curves
of a surface $S$ with marked points.} 
\label{simple_holed}
\end{center}
\end{figure}

\begin{definition}[{{\it Category of topological surfaces}}]
A category $\Top_{g,n}$ of {\it topological surfaces} is one whose {\bf objects} are
 topological surfaces of genus $g$ with $n$ marked points and whose
{\bf morphisms} are elements of mapping class group $MCG(S)$.

\end{definition} 

\begin{theorem}
\label{top}
For any pair $(g,n)$ such that  $2g+n>2$
the category $\ADP_{g,n}$ contains a subcategory $\TopDP_{g,n}$ which is isomorphic to $\Top_{g,n}$. 

\end{theorem}

\begin{proof}
Consider an admissible double pants decomposition $DP$ without double curves.
Consider the orbit $MCG(DP)$ of $DP$ 
under the action of the mapping class group. It follows from Lemma~\ref{ideal} that for $g_1,g_2\in MCG$, $g_1\ne g_2$ 
one has $g_1(DP)\ne g_2(DP)$.  Let $\TopDP_{g,n}$ be a subcategory of $\ADP_{g,n}$ such that the objects of 
$\TopDP_{g,n}$ are elements of the orbit $MCG(DP)$. 
The assumptions of Lemma~\ref{ideal} imply that the objects of $\TopDP_{g,n}$ are in one-to-one correspondence with the objects of 
$\Top_{g,n}$.
Furthermore, Theorem~\ref{trans} implies that for each two objects $x,y\in \TopDP_{g,n}$ there exists a morphism $x\to y$.
So, the morphisms of $\Top_{g,n}$ and $\TopDP_{g,n}$ are in one-to-one correspondence and we have an equivalence of two categories.


\end{proof}

\begin{remark}
The special choice of double pants decomposition without double curves 
in  Theorem~\ref{top} is indispensable: if $c_1,\dots,c_k\in DP$ are double curves
then all Dehn twists along $c_1,\dots,c_k$ preserve $DP$ 
and the orbit $MCG(DP)$ gives a  subcategory of $\ADP_{g,n}$ isomorphic to some proper subcategory  of $\Top_{g,n}$.

\end{remark}

\end{document}